\documentclass{amsart}
\usepackage[usenames,dvipsnames]{color}
\usepackage{anysize}
\marginsize{3cm}{3cm}{4cm}{4cm}

\newtheorem{thm}{Theorem}[section]
\newtheorem{cor}[thm]{Corollary}
\newtheorem{lem}[thm]{Lemma}
\newtheorem{prop}[thm]{Proposition}
\theoremstyle{definition}
\newtheorem{defn}[thm]{Definition}
\theoremstyle{remark}
\newtheorem{rem}[thm]{Remark}
\numberwithin{equation}{section}

\newcommand{\norm}[1]{\left\Vert#1\right\Vert}
\newcommand{\prode}[1]{\left\langle#1\right\rangle}

\newcommand{\paren}[1]{\left(#1\right)}

\newcommand{\Hil}{\mathcal{H}}

\newcommand{\Vsb}{\mathcal{V}}
\newcommand{\Wsb}{\mathcal{W}}

\newcommand{\N}{\mathbb{N}}

\newcommand{\e}{\epsilon}

\begin{document}

\title[Approximate oblique duality for fusion frames]{Approximate oblique duality for fusion frames}

\author[Jorge P. Díaz]{Jorge P. D\'{i}az$^{1}$}%

\author[Sigrid B. Heineken]{Sigrid B. Heineken$^{2}$}%
%%{Sigrid B. Heineken\\Departamento de Matem\'atica \\Facultad de Ciencias Exactas y Naturales\\
%%Universidad de Buenos Aires\\ Pabell\'on I\\ Pabell\'on I\\
%%C1428EGA C.A.B.A.\\ Argentina\\ IMAS, UBA-CONICET\\Argentina.}
%%\email{sheinek@dm.uba.ar}%

\author[Patricia M. Morillas]{Patricia M. Morillas$^{1,*}$\\\\ $^{1}$\textit{I\lowercase{nstituto de }M\lowercase{atem\'{a}tica }A\lowercase{plicada} S\lowercase{an} L\lowercase{uis, }UNSL-CONICET, E\lowercase{j\'{e}rcito de los }A\lowercase{ndes 950, 5700 }S\lowercase{an} L\lowercase{uis,} A\lowercase{rgentina} \\ $^2$ D\lowercase{epartamento de} M\lowercase{atem\'atica}, FCE\lowercase{y}N, U\lowercase{niversidad de }B\lowercase{uenos} A\lowercase{ires}, P\lowercase{abell\'on} I, C\lowercase{iudad }U\lowercase{niversitaria}, IMAS, UBA-CONICET, C1428EGA C.A.B.A., A\lowercase{rgentina}}}
%%\address{Patricia M. Morillas\\Instituto de Matem\'{a}tica Aplicada San Luis\\
%%         UNSL-CONICET\\
%%         and Departamento de Matem\'{a}tica\\
%%         FCFMyN, UNSL\\
%%         Ej\'{e}rcito de los Andes 950, 5700 San Luis, Argentina}%
%%\email{morillas@unsl.edu.ar}%
%%
%
%%\date{}%
%%\dedicatory{}%
%%\commby{}%
%
\thanks{* Corresponding author. E-mail address: morillas.unsl@gmail.com\\
\textit{E-mail addresses:} jpdiaz1179@gmail.com (J. P. Díaz),
sheinek@dm.uba.ar (S. B. Heineken), morillas.unsl@gmail.com (P. M.
Morillas).}
%% ---------------------------------------------------------------
%

\begin{abstract}
Fusion frames are a convenient tool in applications where we deal with a large amount of data or when a combination of local data is needed. Oblique dual fusion frames are suitable in situations where the analysis for the data and its subsequent synthesis have to be implemented in different subspaces of a Hilbert space. These procedures of analysis and synthesis are in general not exact, and also there are circumstances where the exact dual is not available or it is necessary to improve its properties. To resolve these questions we introduce the concept of approximate oblique dual fusion frame, and in particular of approximate oblique dual fusion frame system. We study their properties. We give the relation to approximate oblique dual frames. We provide methods for obtaining them. We show how to construct other duals from a given one that give reconstructions errors as small as we want.

\bigskip

\bigskip

{\bf Keywords:} Fusion frames; Hilbert spaces; Approximate duals; Oblique duals; Analysis and synthesis procedures.

\medskip

{\bf Mathematics Subject Classification 2020:} 42C15, 42C40,
46C05.

\end{abstract}

\maketitle

% ----------------------------------------------------------------
\section{Introduction}

A fusion frame \cite{Casazza-Kutyniok (2004), Casazza-Kutyniok-Li (2008), KBCHHMS (2023)} is a collection of closed subspaces of a separable Hilbert space together with a sequence of weights. They are a generalization of frames \cite{Christensen (2016), Kovacevic-Chebira (2008)}. Frames and fusion frames play a central role in applications where it is required to have representations of vectors in a separable Hilbert space that are redundant (or equivalently, not unique), robust to noises
and erasures, and stable against perturbations. Fusion frames turned out to be particularly useful in situations with distributed processing requirements that cannot be suitably handled by classical frames. For example, they appear as a tool to overcome the difficulty of having a large amount of data. They are also used when there are collector nodes that provide overlapping local data about some object, and this data is gathered to obtain information about the object. Both situations share the following common structure: the \emph{synthesis process} that fuses each portion of data or local data that is obtained by the \emph{analysis process} in the form of fusion frame coefficients. The analysis and synthesis are implemented with a pair of dual fusion frames \cite{Heineken-Morillas-Benavente-Zakowicz (2014), Heineken-Morillas (2014)}. Oblique dual fusion frames \cite{Heineken-Morillas (2018)} are used when we need to work in different subspaces for the analysis and the synthesis procedures. Applications of fusion frames include wireless sensor networks \cite{Iyengar-Brooks 2012}, visual and hearing systems \cite{Rozell-Johnson (2006)}, geophones in geophysics \cite{Craig-Genter (2006)}, quantum communications \cite{Bodmann-Kribs-Paulsen (2007)}, filter banks \cite{Chebira-Fickus-Mixon (2011)}, signal processing \cite{Chen-Powell (2016), Gao-Krahmer-Powell (2021), Xia-Li (2017)}, image processing \cite{Li-Yao-Yi (2012)}, and much more.

In the implementation of oblique dual fusion frames we have limitations and computational difficulties to overcome:

- In some cases, the subspaces are determined by the particular problem and the oblique dual fusion frame is unique. This unique oblique dual fusion frame could not have the adequate properties and it is necessary to improve it.

- In a computer, we work with finite sequences of numbers, hence it is mandatory to truncate the series of the fusion frame representations. 

- The fusion frame coefficients come from measures and thus they are not exact. They are also subject to quantization errors.

- The duals can in general only be computed approximately. Moreover, it may be not possible to
give an analytic expression of it.

In this paper, we attack these difficulties introducing approximate oblique duality for fusion frames, and in particular for fusion frames systems, i.e. fusion frames with generating frame systems in each subspace. In general, the duality theory for
fusion frames is much more elaborate than for the standard frames. This is why in this work we need to use additional tools to overcome the obstacles that emerge and the results are more intricate.

We organize the paper in the following way:

In Section 2, we present the notation and preliminary results about operators, oblique projections, frames and fusion frames.

In Section 3, we introduce approximate oblique duality for fusion frames and fusion frame systems. We obtain some fundamental properties. We establish the link to approximate oblique dual frames. We set requirements for proper analysis and approximate synthesis procedures, namely, uniqueness and approximate consistency for the sampling. We analyse these requirements in relation to approximate oblique duality.

In Section 4, we give methods for obtaining approximate oblique dual fusion frames and approximate oblique dual fusion
frame systems. We study the behaviour of the approximate oblique duals under the perturbations of local frames in each subspace and of the entire fusion frame.

In Section 5, starting from a fusion frame and an approximate oblique dual of it, we construct a sequence of other approximate oblique duals that give reconstructions as close as we desire. We also construct (exact) oblique dual fusion frames from approximate oblique dual fusion frames.

% ----------------------------------------------------------------
\section{Preliminaries}

We first introduce the notation and later we will review some
definitions and present properties that we will need later.

\subsection{Notation}

Consider $\mathcal{H}, \mathcal{K}$ separable Hilbert spaces over
$\mathbb{F}=\mathbb{R}$ or $\mathbb{F}=\mathbb{C}$. We denote by $B(\mathcal{H},\mathcal{K})$ the space of
bounded operators from $\mathcal{H}$ to $\mathcal{K}$. For $A \in
B(\mathcal{H},\mathcal{K})$ we denote the image, the null space and
the adjoint of $A$ by  ${\rm Im}(A)$, ${\rm Ker}(A)$ and $A^{\ast}$,
respectively. If $A$ has closed range we write $A^{\dagger}$ for the Moore-Penrose pseudo-inverse of $A.$ The
inner product and the norm in $\mathcal{H}$ will be denoted by
$\langle \cdot ,\cdot \rangle$ and $\|\cdot\|$, respectively.

Different from what occurs with frames, fusion frames do not behave well under the action of operators \cite{Ruiz-Stojanoff (2008)}. For this reason we recall the concept of reduced minimum modulus that will help us with this question when it arises. If $A \in B(\mathcal{H},\mathcal{K})$, $\gamma(A)$ is the reduced minimum modulus of $A$, i.e.,
$\gamma(A)=\text{inf}\{\|Ax\|:\|x\|=1, x \in {\rm Ker}(A)^{\perp}\}.$ We have $\gamma(A) > 0$ if and only if $A$ has a closed non zero range.

Let $I$ be a countable index set and $\{\mathcal{H}_{i}\}_{i \in I}$
be a sequence of Hilbert spaces. Then

\centerline{$\oplus_{i \in I}\mathcal{H}_{i} = \{(f_{i})_{i \in
I}:f_{i} \in \mathcal{H}_{i} \text{ and } \{\|f_{i}\|\}_{i \in I}
\in \ell^{2}(I)\}$}

\noindent is a Hilbert space with inner product $\langle(f_{i})_{i \in I},(g_{i})_{i
\in I}\rangle=\sum_{i \in I}\langle f_{i}, g_{i}\rangle.$

For $J \subseteq I$ let $\chi_{J} : I \rightarrow \{0, 1\}$ be the
characteristic function of $J.$ We abbreviate
$\chi_{\{j\}}=\chi_{j}$.

Throughout this work, $\mathcal{W}$ and $\mathcal{V}$ will be two closed
subspaces of $\mathcal{H}$.

\subsection{Oblique projections and left inverses}

Assume that $\mathcal{H}=\mathcal{V} \oplus
\mathcal{W}^{\perp} $. By \cite[Lemma 2.1]{Christensen-Eldar (2004)}
this is equivalent to $\mathcal{H}=\mathcal{W} \oplus
\mathcal{V}^{\perp}$. The oblique projection onto $\mathcal{V}$ along
$\mathcal{W}^{\perp}$, is the unique operator that satisfies 

\centerline{$\pi_{\mathcal{V},\mathcal{W}^{\perp}}f=f$, for all $f \in \mathcal{V}$,$\,\,\,\,\,\,\,\,\,\,\pi_{\mathcal{V},\mathcal{W}^{\perp}}f=0$ for
all $f \in \mathcal{W}^{\perp}$.}

\noindent Equivalently, ${\rm
Im}(\pi_{\mathcal{V},\mathcal{W}^{\perp}})=\mathcal{V}$ and ${\rm
Ker}(\pi_{\mathcal{V},\mathcal{W}^{\perp}})=\mathcal{W}^{\perp}$. We have $\left(\pi_{\Vsb\Wsb^{\perp}}\right)^{\ast} =
\pi_{\Wsb\Vsb^{\perp}}$ and
$(\pi_{\Vsb\Wsb^{\perp}})_{\mid\Vsb}=I_{\Vsb}$. If
$\mathcal{W}=\mathcal{V}$ we obtain the orthogonal projection onto
$\mathcal{W}$, denoted by $\pi_{\mathcal{W}}$.

For $\mathcal{H}=\mathcal{V} \oplus
\mathcal{W}^{\perp}$ and $T \in B(\mathcal{H},\mathcal{K})$ such that
${\rm Ker}(T)=\mathcal{W}^{\perp}$, we denote by $\mathfrak{L}_{T,\epsilon}^{\mathcal{V},\mathcal{W}^{\perp}}$ the set of $\epsilon$-approximate oblique left inverses of $T$ on $\mathcal{V}$
along $\mathcal{W}^{\perp}$ which image is equal to $\mathcal{V}$,
i. e.,

\centerline{$\mathfrak{L}_{T,\epsilon}^{\mathcal{V},\mathcal{W}^{\perp}}=\{U
\in
B(\mathcal{K},\mathcal{H}):||UT-\pi_{\mathcal{V},\mathcal{W}^{\perp}}||
\leq \epsilon \text{ and } {\rm Im}(U)=\mathcal{V}\}$.}

If $\mathcal{V}=\mathcal{W}$, we write
$\mathfrak{L}_{T,\epsilon}^{\mathcal{W}}=\{U\in
B(\mathcal{K},\mathcal{H}):||UT-\pi_{\mathcal{W}}||
\leq \epsilon \text{ and } {\rm Im}(U)=\mathcal{W}\}$, for the set of $\epsilon$-approximate left inverses of $T$ on
$\mathcal{W}$ such that ${\rm Im}(U)=\mathcal{W}$.

\subsection{Frames}

Frames have been introduced by Duffin and Schaeffer in
\cite{Duffin-Schaeffer (1952)}. Frames give representations of elements in a Hilbert space that are not necessarily unique. This redundancy makes them
useful for applications in many different areas, such as signal
processing, coding theory, communication theory, sampling theory and
quantum information.

The concept of frame for a closed subspace of
$\mathcal{H}$ is the following:

\begin{defn}\label{D frame}
Let $\{f_{i}\}_{i \in I} \subset \mathcal{W}$. Then $\{f_{i}\}_{i \in
I}$ is a \emph{frame} for $\mathcal{W}$, if there exist constants $0
< \alpha \leq \beta < \infty$ such that
\begin{equation}\label{E cond frame}
\alpha\|f\|^{2} \leq \sum_{i \in I}|\langle f,f_{i}\rangle |^{2}
\leq \beta\|f\|^{2}  \text{ for all $f \in \mathcal{W}$.}
\end{equation}
\end{defn}

If the right inequality in (\ref{E cond frame}) holds, then
$\{f_{i}\}_{i \in I}$ is a {\it Bessel sequence} for $\mathcal{W}$. We call the constants $\alpha$ and $\beta$ {\it frame bounds}. When $\alpha=\beta,$ we say that $\{f_{i}\}_{i \in I}$ is an {\it
$\alpha$-tight frame}, and if $\alpha=\beta=1$ that it is a {\it Parseval
frame} for $\mathcal{W}$.

We have the following bounded operators associated to a Bessel sequence $\mathcal{F}=\{f_{i}\}_{i \in I} \subseteq \mathcal{W}$:

- the {\it synthesis operator} $T_{\mathcal{F}}:\ell^2(I)\rightarrow \mathcal{H},$
$T_{\mathcal{F}}\{c_i\}_{i\in I}=\sum_{i\in I}c_if_i,$

- the {\it analysis operator} $T_{\mathcal{F}}^{\ast}: \mathcal{H}\rightarrow \ell^2(I)$,
$T_{\mathcal{F}}^{\ast}f=\{\langle f,f_i\rangle \}_{i\in I},$

- the {\it frame operator} $S_{\mathcal{F}}=T_{\mathcal{F}}T_{\mathcal{F}}^{\ast}$.

\noindent We say that a Bessel sequence $\mathcal{F}=\{f_{i}\}_{i \in I}$ for
$\mathcal{W}$ is a frame for $\mathcal{W}$ if and only ${\rm
Im}(T_{\mathcal{F}})=\mathcal{W}$, or equivalently,
${S_{\mathcal{F}}}_{|\Wsb}$ is invertible.
Additionally, $\mathcal{F}$ is an $\alpha$-tight frame for
$\mathcal{W}$ if and only if $S_{\mathcal{F}}=\alpha
\pi_{\mathcal{W}}$.

When $\mathcal{W}$ is finite-dimensional we will contemplate
finite frames for it, i.e., frames with a finite number of elements.
We point out that if $\text{dim}(\mathcal{W})<\infty$ and
$|I|<\infty$, then $\{f_{i}\}_{i \in I} \subset \mathcal{H}$ is a
frame for $\mathcal{W}$ if and only if $\text{span}\{f_{i}\}_{i \in
I}=\mathcal{W}$.

We refer the reader to
\cite{Christensen (2016), Kovacevic-Chebira
(2008)} for more details about frames. Oblique dual frames  \cite{Eldar (2003a),
Eldar (2003b), Eldar-Werther (2005), Li-Lian
(2011), Xiao-Zhu-Zeng (2013)} are defined in the following way:

\begin{defn}\label{D oblique dual frame}
Let $\mathcal{H}=\mathcal{V}\oplus
\mathcal{W}^{\perp}$. Let $\mathcal{F}=\{f_{i}\}_{i \in I}$ be a
frame for $\mathcal{W}$ and $\mathcal{G}=\{g_{i}\}_{i \in I}$ be a
frame for $\mathcal{V}$. If $T_{\mathcal{G}}T_{\mathcal{F}}^{\ast}=\pi_{\mathcal{V},\mathcal{W}^{\perp}},$ we say that $\mathcal{G}$ is an {\em oblique dual frame} of $\mathcal{F}$ on $\mathcal{V}$.
\end{defn}

We call
$\{\pi_{\mathcal{V},\mathcal{W}^{\perp}}S_{\mathcal{F}}^{\dagger}f_{i}\}_{i
\in I}$ the \emph{canonical oblique dual frame} of $\{f_{i}\}_{i
\in I}$ on $\mathcal{V}$.

\begin{rem}
If $\mathcal{V}=\mathcal{W}$ we have the classical duals and we say dual frame instead of oblique dual frame on
$\mathcal{W}$.
\end{rem}

A frame for $\mathcal{W}$ which is also a basis is called a {\it Riesz basis} for $\mathcal{W}$. Note that a Riesz basis has a unique dual,
the canonical one.

\subsection{Fusion frames}

In \cite{Casazza-Kutyniok (2004)} fusion frames were introduced under the name of \textit{frames of
subspaces}. They are an excellent instrument for dealing with different situations in sensor networking, distributed processing, etc.
In the following we will work with fusion frames for closed
subspaces of $\mathcal{H}$.

Suppose that $\{W_{i}\}_{i \in I}$ is a family of closed subspaces in
$\mathcal{W}$, with $\{w_{i}\}_{i \in I}$ a family of weights, i.e.,
$w_{i}
> 0$ for all $i \in I$. We write $\mathbf{W}$ for $\{W_{i}\}_{i \in I}$, $\mathbf{w}$ for $\{w_{i}\}_{i
\in I}$, and $(\mathbf{W},\mathbf{w})$ for
$\{(W_i,w_{i})\}_{i \in I}$. For $A \in B(\mathcal{H},\mathcal{K})$ we denote $\{(AW_i, w_{i})\}_{i \in I}$ with $(A\mathbf{W},\mathbf{w})$.

Let us consider the Hilbert space $\mathcal{K}_{\mathbf{W}}=\oplus_{i
\in I}W_{i}$.

\begin{defn}\label{D fusion frame}
We say that $(\mathbf{W},\mathbf{w})$ is a {\it fusion frame} for
$\mathcal{W}$, if there exist constants $0 < \alpha \leq \beta <
\infty$ such that
\begin{equation}\label{E cond ff}
\alpha\|f\|^{2} \leq \sum_{i \in I}w_{i}^{2}\|\pi_{W_{i}}(f)\|^{2}
\leq \beta\|f\|^{2}  \text{ for all $f \in \mathcal{W}$.}
\end{equation}
\end{defn}

We name $\alpha$ and $\beta$ the \textit{fusion frame bounds}. The collection $(\mathbf{W},\mathbf{w})$ is an $\alpha$-\textit{tight
fusion frame} for $\mathcal{W}$, if in (\ref{E cond ff}) we can take $\alpha$ and $\beta$ such that $\alpha =
\beta$, and a \textit{Parseval fusion frame} for $\mathcal{W}$
if $\alpha = \beta = 1.$ The family $(\mathbf{W},\mathbf{w})$ is called a \textit{fusion Bessel sequence} for $\mathcal{W}$ with
fusion Bessel bound $\beta$  if it has
an upper fusion frame bound, but not necessarily a lower bound. When $w_{i}=c$ for all $i\in I,$ we use the notation
$\mathbf{w}=c$. If $\mathcal{W}$ is the direct sum of the $W_i$ we call $(\mathbf{W},{\bf w})$ a \textit{Riesz fusion basis} for
$\mathcal{W}$. A fusion frame that is not a Riesz
fusion basis is an overcomplete fusion frame. A fusion frame
$(\mathbf{W},1)$ is called an \textit{orthonormal fusion basis} for
$\mathcal{W}$ if $\mathcal{W}$ is the orthogonal sum of the
subspaces $W_{i}.$

Given a fusion Bessel sequence $(\mathbf{W},\mathbf{w})$ for
$\mathcal{W},$ we  associate the following bounded operators:

- the \textit{synthesis operator} $T_{\mathbf{W},\mathbf{w}} : \mathcal{K}_{\mathbf{W}}
\rightarrow  \mathcal{H},$
$\,\,\,T_{\mathbf{W},\mathbf{w}}(f_{i})_{i \in I}=\sum_{i \in
  I}w_{i}f_{i},$

- the \textit{analysis operator} $T_{\mathbf{W},\mathbf{w}}^{\ast} : \mathcal{H} \rightarrow
\mathcal{K}_{\mathbf{W}},$
$\,\,\,T_{\mathbf{W},\mathbf{w}}^{\ast}f=(w_{i}\pi_{W_{i}}(f))_{i \in
I}$,

- the \textit{fusion frame operator} $S_{\mathbf{W},\mathbf{w}}=T_{\mathbf{W},\mathbf{w}}T_{\mathbf{W},\mathbf{w}}^{\ast}.$

\noindent Similar to what happens for frames, $(\mathbf{W},\mathbf{w})$ is a fusion Bessel sequence for $\mathcal{W}$ if and only if
$T_{\mathbf{W},\mathbf{w}}$ is a well defined bounded linear
operator. A fusion Bessel sequence $(\mathbf{W},\mathbf{w})$  for
$\mathcal{W}$ is a fusion frame for $\mathcal{W}$ if and only if
${\rm Im}(T_{\mathbf{W},\mathbf{w}})=\mathcal{W}$, or equivalently,
${S_{\mathbf{W},\mathbf{w}}}_{|\Wsb}$ is invertible. Furthermore, $(\mathbf{W},\mathbf{w})$ is an
$\alpha$-tight fusion frame for $\mathcal{W}$ if and only if
$S_{\mathbf{W},\mathbf{w}}=\alpha \pi_{\mathcal{W}}.$

For finite-dimensional subspaces $\mathcal{W}$ we will consider
finite fusion frames, i.e., fusion frames with a finite set of
indices. Observe that if $\text{dim}(\mathcal{W})<\infty$ and
$|I|<\infty$, then $(\mathbf{W},\mathbf{w})$ is a frame for
$\mathcal{W}$ if and only $\text{span}\cup_{i \in
I}W_{i}=\mathcal{W}$.

As was noted in \cite[Remark
2.4]{Heineken-Morillas-Benavente-Zakowicz (2014)}, if $w \in
\ell^{2}(I)$, then $(\mathcal{W},w)$ is a fusion Bessel sequence for
$\mathcal{H}$, and if $T_{\mathcal{W},w}$ is a well defined bounded
operator then $w \in \ell^{\infty}(I)$. In view of this, in the
sequel we suppose that each family of weights is in
$\ell^{\infty}(I)$.

Fusion frames allow local processing in each of the
subspaces. Considering this, having a set of local frames for its
subspaces is convenient.

\begin{defn}
Let $(\mathbf{W},{\bf w})$ be a fusion frame (\emph{fusion Bessel sequence}) for $\mathcal{W}$, and let $\{f_{i,l}\}_{l\in L_i}$ be a
frame for $W_{i}$ for $i \in I$. Then $\{(W_i,w_i,\{f_{i,l}\}_{l\in
L_i})\}_{i \in I}$ is called a \emph{fusion frame system}
(\emph{fusion Bessel system}) for $\mathcal{W}$.
\end{defn}

In what follows we will use the notation
$\mathcal{F}_{i}=\{f_{i,l}\}_{l \in L_i}$,
$\mathcal{F}=\{\mathcal{F}_{i}\}_{i \in I},$ ${\bf
w}\mathcal{F}=\{w_{i}\mathcal{F}_{i}\}_{i \in I},$ and we write
$(\mathbf{W}, {\bf w}, \mathcal{F})$ for
$\{(W_i,w_i,\{f_{i,l}\}_{l\in L_i})\}_{i \in I}$. If $A \in
B(\mathcal{H},\mathcal{K})$ we write $A\mathcal{F}$ for
$\{\{A{f}_{i,l}\}_{l\in L_i}\}_{i \in I}$ and $A\mathcal{F}_{i}$ for
$\{Tf_{i,l}\}_{l \in L_i}$.

\begin{thm}\cite[Theorem 3.2]{Casazza-Kutyniok (2004)}\label{T wF marco sii WwF fusion frame system}
Given $(\mathbf{W},\mathbf{w})$, let $\mathcal{F}_{i}$ be a frame for
$W_{i}$ with frame bounds $\alpha_{i}, \beta_{i}$ such that $0 <
\alpha ={\rm inf}_{i \in I}\alpha_{i} \leq {\rm sup}_{i \in
I}\beta_{i}= \beta < \infty$. The following assertions are
equivalents:
\begin{enumerate}
  \item $\mathbf{w}\mathcal{F}$ is a frame for $\mathcal{W}$.
  \item $(\mathbf{W},\mathbf{w})$ is a fusion frame for
  $\mathcal{W}$.
\end{enumerate}
If $(\mathbf{W},\mathbf{w})$ is a fusion frame for $\mathcal{W}$
with fusion frame bounds $\gamma$ and $\delta$, then
$\mathbf{w}\mathcal{F}$ is a frame for $\mathcal{W}$ with frame
bounds $\alpha \gamma$ and $\beta \delta$. If
$\mathbf{w}\mathcal{F}$ is a frame for $\mathcal{W}$ with frame
bounds $\gamma$ and $\delta$, then $(\mathbf{W}, {\bf w})$ is a
fusion frame for $\mathcal{W}$ with fusion frame bounds
$\frac{\gamma}{\beta}$ and $\frac{\delta}{\alpha}$. The previous assertions are valid if we replace fusion frame by
fusion Bessel sequence and consider only the upper bounds.
\end{thm}
We
refer the reader to \cite{Casazza-Kutyniok (2004),
Casazza-Kutyniok-Li (2008), Casazza-Fickus-Heinecke-Wang-Zhou (2012)} for more details about fusion frames and fusion frame systems.

Consider the selfadjoint
operator $M_{J,\mathbf{W}} : \mathcal{K}_{\mathbf{W}} \rightarrow
\mathcal{K}_{\mathbf{W}}, M_{J,\mathbf{W}}(f_i)_{i \in
I}=(\chi_{J}(i)f_{i})_{i \in I}.$ We only write $M_{J}$ if it clear
to which $\mathbf{W}$ we refer to. We write $M_{\{j\},\mathbf{W}}=M_{j,\mathbf{W}}$ and $M_{\{j\}}=M_{j}$. We note that if $A\in B(\mathcal{H},\mathcal{K})$, then $AW_{j}=\{A T_{\mathbf{W},{\bf
w}} M_j (f_{i})_{i \in I} : (f_{i})_{i \in I} \in
\mathcal{K}_{\mathbf{W}}\}$.

\begin{defn}Let $Q \in B(\mathcal{K}_{\mathbf{W}}, \mathcal{K}_{\mathbf{V}})$.
\begin{enumerate}
\item If $QM_{j,\mathbf{W}}\mathcal{K}_{\mathbf{W}}
\subseteq M_{j, {\bf V}}\mathcal{K}_{\mathbf{V}}$ for each $j \in
I,$ we call $Q$ \emph{block diagonal}.
\item If $QM_{j,\mathbf{W}}\mathcal{K}_{\mathbf{W}}
= M_{j, {\bf V}}\mathcal{K}_{\mathbf{V}}$ for each $j \in I,$ we call $Q$ {\em component preserving}.
\end{enumerate}
\end{defn}
Observe that $Q$ is block diagonal if and only if
$QM_{J,\mathbf{W}}=M_{J,\mathbf{V}}Q$ for each $J \subseteq I$, i.e. $QM_{j,\mathbf{W}}=M_{j, {\bf V}}Q$ for each $j \in
I$. Note that if $Q$ is block diagonal, then $Q^{\ast}$ is
block diagonal.

% ----------------------------------------------------------------
\section{Approximate oblique dual fusion frames and fusion frame systems}

Oblique duality comes up in many situations where we need to reconstruct a signal in a
different subspace than the one where the analysis is done. On the other hand, sometimes it is difficult to have an expression for an oblique dual fusion frame or there only exists one oblique dual fusion frame which lacks the attributes we desire. To address these issues we will now introduce the definitions of approximate oblique dual fusion frame and of approximate oblique dual fusion frame system.

\begin{defn}\label{D oblique fusion frame dual}
Let $\mathcal{H}=\mathcal{V}\oplus
\mathcal{W}^{\perp}$. Let $(\mathbf{W},{\bf w})$ be a fusion frame
for $\mathcal{W}$ and $({\bf V},{\bf v})$ be a fusion frame for
$\mathcal{V}$. Let $\epsilon\geq 0.$ We say that $(\mathbf{V},{\bf v})$ is an
$\epsilon$-approximate oblique dual fusion frame of
$(\mathbf{W},{\bf w})$ on $\mathcal{V}$ if there exists $Q \in
B(\mathcal{K}_{\mathbf{W}}, \mathcal{K}_{\mathbf{V}})$ such that
\begin{equation}\label{E TvQTw*=I}
||T_{{\bf V},{\bf v}}QT^{\ast}_{\mathbf{W},{\bf
w}}-\pi_{\mathcal{V},\mathcal{W}^{\perp}}|| \leq \epsilon.
\end{equation}
\end{defn}
When we want to refer explicitly to a particular $Q$ we say that $({\bf V},{\bf
v})$ is an $\epsilon$-approximate oblique $Q$-dual fusion frame of $(\mathbf{W},{\bf w})$.
Observe that if $({\bf V},{\bf v})$ is an $\e$-approximate oblique $Q$-dual fusion frame
of $(\mathbf{W},{\bf w})$ on $\mathcal{V}$, then $({\bf W},{\bf w})$
is an $\e$-approximate oblique $Q^{\ast}$-dual fusion frame of $(\mathbf{V},{\bf v})$ on
$\mathcal{W}$. If in Definition~\ref{D oblique fusion frame dual}
$Q$ is block diagonal (component preserving) we say that $({\bf
V},{\bf v})$ is  an $\epsilon$-approximate oblique \emph{block diagonal dual fusion frame}
($\epsilon$-approximate oblique \emph{component preserving dual fusion frame}) of
$(\mathbf{W},{\bf w})$. We refer to \cite{KBCHHMS (2023), Heineken-Morillas (2018)} for a more detailed discussion about the inclusion of an operator $Q$ and its role, which applies also here in this more general setting.
\begin{rem}\label{Obs definicion de mdoa}
\begin{enumerate}
\item[(i)]
For the particular case $ \e = 0 $, $(\mathbf{W},{\bf w})$ and $(\mathbf{V},{\bf v})$ are
\emph{ oblique dual fusion frames} \cite{Heineken-Morillas (2018)}. Moreover, if $ \Vsb = \Wsb = \Hil $ and $ \e = 0 $, then $(\mathbf{W},{\bf
w})$ and $(\mathbf{V},{\bf v})$ are \emph{dual fusion frames} as
defined in \cite{Heineken-Morillas-Benavente-Zakowicz (2014),
Heineken-Morillas (2014)}.

\item[(ii)]
For $ \e < 1 $, it is enough that $(\mathbf{W},{\bf w})$ and $(\mathbf{V},{\bf v})$ are fusion Bessel sequences because $
\mathcal{R}( T_{\mathbf{W},{\bf w}}) = \Wsb $ by (ii) and so $(\mathbf{W},{\bf w})$ is a fusion frame for $ \Wsb $. Similarly $(\mathbf{V},{\bf v})$ is a fusion frame for $ \Vsb $.
\end{enumerate}
\end{rem}
The following characterization lemma has a proof similar to the one of
\cite[Lemma 4.1]{Diaz-Heineken-Morillas (2023)}.
\begin{lem}\label{L equivalencias}
Let $\mathcal{H}=\mathcal{V}\oplus
\mathcal{W}^{\perp} $. Let $(\mathbf{W}, {\bf w})$ be a fusion Bessel sequence for $\mathcal{W}$, $(\mathbf{V}, {\bf v})$ be a
fusion Bessel sequence for $\mathcal{V}$, and let $Q \in
B(\mathcal{K}_{\mathbf{W}}, \mathcal{K}_{\mathbf{V}})$. Let $\e \geq 0$. Then the
following statements are equivalent:
\begin{enumerate}
  \item[(i)] $||\pi_{\mathcal{V},{\mathcal{W}}^\bot}f-T_{{\bf V},{\bf v}}Q T^{\ast}_{{\bf W},{\bf w}}f|| \leq \epsilon ||f||$ for all $f\in \mathcal{H}.$
  \item[(ii)] $||\pi_{\mathcal{W},{\mathcal{V}}^\bot}f-T_{{\bf W},{\bf w}}Q^* T^{\ast}_{{\bf V},{\bf v}}f||  \leq \epsilon ||f||$ for all $f\in \mathcal{H}.$
  \item[(iii)] $|\prode{\pi_{\mathcal{W},{\mathcal{V}}^\bot} f,g}-\prode{Q^*T^{\ast}_{{\bf V},{\bf v}}f,T_{{\bf W},{\bf w}}^{\ast}g}| \leq \epsilon ||f|| ||g||$
  for all $f, g \in \mathcal{H}$.
 \item[(iv)] $|\prode{\pi_{\mathcal{V},{\mathcal{W}}^\bot} f,g}-\prode{QT^{\ast}_{{\bf W},{\bf w}}f,T_{{\bf V},{\bf v}}^{\ast}g}| \leq \epsilon ||f|| ||g||$
  for all $f, g \in \mathcal{H}$.
 \end{enumerate}
If $(\mathbf{W}, {\bf w})$ is a fusion frame for $\mathcal{W}$, $(\mathbf{V}, {\bf v})$ is a fusion frame for $\mathcal{V}$, and any of these equivalent conditions is satisfied, then
$(\mathbf{V}, {\bf v})$  is an $\epsilon$-approximate oblique $Q$-dual fusion frame of
$(\mathbf{W}, {\bf w})$ on $\mathcal{V},$ and $(\mathbf{W}, {\bf
w})$ is an $\epsilon$-approximate oblique $Q^{\ast}$-dual fusion frame of $(\mathbf{V}, {\bf
v})$ on $\mathcal{W}.$
\end{lem}
Taking into account Lemma~\ref{L equivalencias}(i), for a block diagonal $Q$, (\ref{E TvQTw*=I}) yields the next simpler approximate reconstruction formula:

\centerline{$||\pi_{\mathcal{V},\mathcal{W}^{\perp}}f-\sum_{j \in
I}v_{j}w_{j}Q_{j}f|| \leq \epsilon||f||~,~~\forall f \in
\mathcal{H}$,}

\noindent where $Q_{j} \in B\!\paren{\mathcal{H}, V_{j}}$ is given by $Q_{j}f=(QM_{j}(\pi_{W_{i}}f)_{i \in I})_{j}$.

We will now define and develop the concept of oblique
dual fusion frame systems in order to consider local processing in the setting described previously. We will use the
following operator, which was introduced in \cite{Heineken-Morillas
(2014)}, and links the synthesis
operator of a fusion frame system and the synthesis operator of its
associated frame.

Assume that $(\mathbf{W}, {\bf w})$ is a fusion Bessel sequence for
$\mathcal{W}$ and $\mathcal{F}_{i}$ be a frame for $W_i$ with frame
bounds $\alpha_{i}, \beta_{i}$ such that ${\rm sup}_{i \in
I}\beta_{i}= \beta < \infty$. We consider the operator

\centerline{$C_{\mathcal{F}}: \oplus_{i \in I}
\ell^2(L_{i})\rightarrow \mathcal{K}_{\mathbf{W}},\,\,
C_{\mathcal{F}}((x_{i,l})_{l \in L_i})_{i \in
I}=(T_{\mathcal{F}_{i}}(x_{i,l})_{l \in L_i})_{i \in I}.$}

Observe that $C_{\mathcal{F}}$ is a surjective bounded operator and
$||C_{\mathcal{F}}|| \leq \beta$. The adjoint operator is
$C_{\mathcal{F}}^{\ast}: \mathcal{K}_{\mathbf{W}} \rightarrow \oplus_{i
\in I} \ell^2(L_{i})$, with $C_{\mathcal{F}}^{\ast}(g_{i})_{i \in
I}=(T_{\mathcal{F}_{i}}^{\ast}g_{i})_{i \in I}$ and it holds
$||C_{\mathcal{F}}^{\ast}(g_{i})_{i \in I}|| \leq \beta||(g_{i})_{i \in
I}||$. If $0 < \alpha ={\rm inf}_{i \in I}\alpha_{i}$, then
$\alpha ||(g_{i})_{i \in I}|| \leq ||C_{\mathcal{F}}^{\ast}(g_{i})_{i
\in I}||$. The bounded left inverses of $C_{\mathcal{F}}^{\ast}$ are
 all $C_{\widetilde{\mathcal{F}}} \in B(\oplus_{i \in I}
\ell^2(L_{i}), \mathcal{K}_{\mathbf{W}})$ such that
$\widetilde{\mathcal{F}}_{i}$ is a dual frame of $\mathcal{F}_{i}$
with an upper frame bound $\widetilde{\beta}_{i}$ that satisfies
$\widetilde{\beta}={\rm sup}_{i \in I}\widetilde{\beta}_{i} <
\infty$. Note that $T_{{\bf w}\mathcal{F}}=T_{\mathbf{W},{\bf
w}}C_{\mathcal{F}}\,\,\text { and }\,\,T_{\mathbf{W},{\bf
w}}=T_{{\bf w}\mathcal{F}}C_{\widetilde{\mathcal{F}}}^{\ast}.$

We now introduce the definition of approximate oblique dual fusion frame systems:
\begin{defn}\label{D oblique dual fusion frame system}
Let $\mathcal{H}=\mathcal{V}\oplus\mathcal{W}^{\perp}$. Let
$(\mathbf{W},{\bf w}, \mathcal{F})$ be a fusion frame system for
$\mathcal{W}$ with upper local frame  bounds $\beta_{i}$ such that
${\rm sup}_{i \in I}\beta_{i} < \infty$, $({\bf V},{\bf v},
\mathcal{G})$ be a fusion frame system for $\mathcal{V}$ with local
upper frame bounds $\widetilde{\beta}_{i}$  such that ${\rm sup}_{i
\in I}\widetilde{\beta}_{i} < \infty$. Let $\mathcal{F}_{i}$ and $\mathcal{G}_{i}$ have the same index set for each $i \in I$. Let $\e \geq 0$. Then
$({\bf V},{\bf v}, \mathcal{G})$ is an $\epsilon$-approximate
oblique dual fusion frame system of $(\mathbf{W},{\bf w},
\mathcal{F})$ on $\mathcal{V}$ if $({\bf V},{\bf v})$ is an
$\epsilon$-approximate oblique
$C_{\mathcal{G}}C_{\mathcal{F}}^{\ast}$-dual fusion frame of $(\mathbf{W},{\bf w})$ on $\mathcal{V}$.
\end{defn}
\begin{rem}\label{R Q Mi sum Wj subset Mi sum Vj}
The operator
$C_{\mathcal{G}}C_{\mathcal{F}}^{\ast}: \mathcal{K}_{\mathbf{W}}
\rightarrow \mathcal{K}_{\mathbf{V}},\,
C_{\mathcal{G}}C_{\mathcal{F}}^{\ast}(f_{i})_{i \in
I}=(T_{\mathcal{G}_{i}}T_{\mathcal{F}_{i}}^{\ast}f_{i})_{i \in I}$ is
block diagonal.
\end{rem}
If $C_{\mathcal{G}}C_{\mathcal{F}}^{\ast}$ in Definition~\ref{D oblique
dual fusion frame system}  is component preserving, we call $({\bf
V},{\bf v}, \mathcal{G})$ an $\epsilon$-\emph{approximate}
\emph{oblique  component preserving dual fusion frame system} of
$(\mathbf{W},{\bf w}, \mathcal{F}).$
\begin{rem}
For $\mathcal{W}=\mathcal{V}$ in Definition~\ref{D oblique
fusion frame dual} and in Definition~\ref{D oblique dual fusion
frame system} we have the concepts of dual fusion frame and dual fusion
frame system, respectively, considered in
\cite{Heineken-Morillas-Benavente-Zakowicz (2014), Heineken-Morillas
(2014)}. We then say that $({\bf V},{\bf v})$ is a
$Q$-dual fusion frame of $(\mathbf{W},{\bf w})$ or that $({\bf
V},{\bf v}, \mathcal{G})$ is a dual fusion frame system of
$(\mathbf{W},{\bf w}, \mathcal{F})$.
\end{rem}
Looking at Definition~\ref{D oblique dual fusion frame
system} and Remark~\ref{R Q Mi sum Wj subset Mi sum Vj} we note that
each $ \e $-approximate oblique dual fusion frame system can be linked to an $ \e $-approximate oblique block diagonal dual fusion frame. Conversely, we can always associate to an  $ \e $-approximate oblique block diagonal dual fusion frame pair an $ \e $-approximate oblique dual fusion frame system pair.
Also, approximate duals can be obtained from approximate oblique duals and viceversa. More precisely, Theorem 4.3, Theorem 4.4, Proposition 5.1, Corollary 5.2, Proposition 5.3 and Corollary 5.4 of \cite{Heineken-Morillas (2018)} can be extended  with the pertinent changes to the approximate case (see also \cite[Section 6]{KBCHHMS (2023)}). In particular we only state the following result which is the extension of \cite[Theorem 4.4]{Heineken-Morillas (2018)} that we will use later:

\begin{thm}\label{T dual fusion frame systems} 
Let $\mathcal{H}=\mathcal{V}\oplus\mathcal{W}^{\perp}$. Let
$(\mathbf{W},{\bf w}, \mathcal{F})$ be a fusion Bessel system for
$\mathcal{W}$ such that $\mathcal{F}_{i}$ has upper frame bound
$\beta_{i}$  with ${\rm sup}_{i \in I}\beta_{i} < \infty$, and let
$({\bf V},{\bf v}, \mathcal{G})$ be a fusion Bessel system for
$\mathcal{V}$  such that $\mathcal{G}_{i}$ has upper frame bound
$\widetilde{\beta}_{i}$ with ${\rm sup}_{i \in
I}\widetilde{\beta}_{i}< \infty$. Let $\e \geq 0$. If
$\mathcal{F}_{i}$ and $\mathcal{G}_{i}$ have the same index set for each $i \in I$, then the
following conditions are equivalent:
  \begin{enumerate}
    \item ${\bf w}\mathcal{F}$ and ${\bf v}\mathcal{G}$ are $\epsilon$-approximate oblique dual frames for $\mathcal{H}.$
    \item  $(\mathbf{V}, {\bf v},\mathcal{G})$ is an $\epsilon$-approximate oblique dual fusion frame system of $(\mathbf{W}, {\bf w},\mathcal{F})$ on $\mathcal{V}$.
  \end{enumerate}
\end{thm}

\subsection{Oblique dual fusion frames from approximate ones}

If in Definition~\ref{D oblique fusion frame dual} $ \e < 1 $, then Neumann's Theorem says that $(T_{{\bf V},{\bf
v}}QT^{\ast}_{\mathbf{W},{\bf w}}) _{|\Vsb} $  is an invertible
operator from $\Vsb$ to $\Vsb$. Note that $(T_{{\bf V},{\bf v}}QT^{\ast}_{\mathbf{W},{\bf
w}}) _{|\Wsb}=(T_{{\bf V},{\bf v}}QT^{\ast}_{\mathbf{W},{\bf w}})
_{|\Vsb} (\pi_{ \Vsb \Wsb^{ \perp } } )_{ \mid \Wsb }$. So by \cite[Lemma 2.1(i)]{Diaz-Heineken-Morillas (2023)}, $(T_{{\bf V},{\bf v}}QT^{\ast}_{\mathbf{W},{\bf w}}) _{|\Wsb}$ is an
invertible operator from $ \Wsb $ to $ \Vsb $. Similarly, $(T_{{\bf
W},{\bf w}}Q^{\ast}T^{\ast}_{\mathbf{V},{\bf v}}) _{|\Wsb} $ and $(T_{{\bf
W},{\bf w}}Q^{\ast}T^{\ast}_{\mathbf{V},{\bf v}}) _{|\Vsb} $ are
invertible operators.
\begin{thm}
Let $\mathcal{H}=\mathcal{V}\oplus
\mathcal{W}^{\perp}$. Let $(\mathbf{W},{\bf w})$ be a fusion frame
for $\mathcal{W}$, $({\bf V},{\bf v})$ be a fusion frame for
$\mathcal{V}$ and $Q \in B(\mathcal{K}_{\mathbf{W}},
\mathcal{K}_{\mathbf{V}})$ such that $(\mathbf{V},{\bf v})$ is an
$\epsilon$-approximate oblique $Q$-dual fusion frame of
$(\mathbf{W},{\bf w})$ on $\mathcal{V}$ with $\epsilon < 1$. Let $A=(T_{{\bf V},{\bf
v}}QT_{\mathbf{W},{\bf w}}^{\ast}) _{|\Vsb})^{ -1 }$ and $\widetilde{\mathbf{v}}$ be a family of weights. If $(A\mathbf{V},\widetilde{{\bf v}})$ is a Bessel sequence for $\mathcal{V}$ and $Q_{A  T_{\mathbf{V},{\bf v}},\widetilde{\mathbf{v}}}$ is a bounded operator, then $(A\mathbf{V},\widetilde{{\bf v}})$ is an oblique $Q_{A  T_{\mathbf{V},{\bf v}},\widetilde{\mathbf{v}}}Q$-dual fusion frame of
$(\mathbf{W},{\bf w})$ on $\mathcal{V}$.
\end{thm}
\begin{proof}
If $ f \in \Vsb $ we have $ f = ((T_{{\bf V},{\bf v}}QT_{\mathbf{W},{\bf w}}^{\ast})
_{|\Vsb})^{ -1 } T_{{\bf V},{\bf v}}QT_{\mathbf{W},{\bf w}}^{\ast} f =T_{A\mathbf{V},\widetilde{{\bf
v}}}Q_{A  T_{\mathbf{V},{\bf v}},\widetilde{\mathbf{v}}}QT^{\ast}_{\mathbf{W},{\bf w}}f$. By Definition~\ref{D oblique fusion frame dual} and
Remark~\ref{Obs definicion de mdoa}(ii), this proves the theorem.
\end{proof}
That $(A\mathbf{V},\widetilde{{\bf v}})$ is a
fusion Bessel sequence for $\mathcal{V}$ and
$Q_{A  T_{\mathbf{V},{\bf v}},\widetilde{\mathbf{v}}}$
is a well defined bounded operator depends on the particular case
we are working with. Nevertheless, we can use Proposition~\ref{P A invertible AV Bessel} (or Lemma~\ref{L V Bessel}) and Lemma~\ref{L Q acotada} below to have sufficient conditions on the weights $\widetilde{\mathbf{v}}$ to get this.
\begin{lem}\label{L V Bessel}
Let $L \in
B(\mathcal{K}_{\mathbf{W}},\mathcal{V})$. Let
$V_{i}=LM_{i}\mathcal{K}_{\mathbf{W}}$ for each $i \in
I$ and $\mathbf{v}$ be a family of weights. If there exists $\delta > 0$ such that $\delta
\leq v_{i}^{-1}(\gamma(LM_{i}))$ for all $i \in I$,
then $(\mathbf{V},\mathbf{v})$ is a fusion Bessel sequence for $\mathcal{V}$ with upper frame bound
$\frac{\|L\|^{2}}{\delta^{2}}$.
\end{lem}
\begin{proof}
Since $\gamma(LM_{i}) > 0$, $V_{i}$ is a closed a subspace of $\mathcal{V}$. Let $\{e_{ij}\}_{j \in J_{i}}$ be an orthonormal basis for $M_{i}\mathcal{K}_{\mathbf{W}}$. By \cite[Corollary 5.3.2]{Christensen (2016)} and the equality $\|(LM_{i})^{\dagger}\|^{-1}=\gamma(LM_{i})$, $\{v_{i}^{-1}Le_{ij}\}_{j \in J_{i}}$ is a frame for $V_{i}$ with lower frame bound $v_{i}^{-2}\gamma(LM_{i})^{2}$. Since $\{e_{ij}\}_{i \in I, j \in J_{i}}$ is an orthonormal basis for $\mathcal{K}_{\mathbf{W}}$, it is easy to see that $\{Le_{ij}\}_{i \in I, j \in J_{i}}$ is a fusion Bessel sequence for $\mathcal{V}$ with upper Bessel bound $\|L\|^{2}$. By Theorem~\ref{T wF marco sii WwF fusion frame system}, $(\mathbf{V},\mathbf{v})$ is a fusion Bessel sequence in $\mathcal{V}$ with upper Bessel bound
$\frac{\|L\|^{2}}{\delta^{2}}$.
\end{proof}
\begin{lem}\label{L Q acotada}
Let $L \in
B(\mathcal{K}_{\mathbf{W}},\mathcal{V})$ be such that
$\gamma(LM_{i}) > 0$ for all $i \in I$. Let
$V_{i}=LM_{i}\mathcal{K}_{\mathbf{W}}$ for each $i \in
I$ and $\mathbf{v}$ be a family of weights. If there exists $\delta > 0$ such that $v_{i} > \delta > 0$ for each $i \in I$, then
\begin{equation}\label{E QA}
Q_{L,\mathbf{v}}: \mathcal{K}_{\mathbf{W}}
\rightarrow \mathcal{K}_{\mathbf{V}}, \,\,\,Q_{L,\mathbf{v}}(f_{j})_{j \in I}=(\frac{1}{v_{i}} L M_i (f_{j})_{j \in I})_{i\in I},
\end{equation}  
is a well defined bounded operator with $\|Q_{L,\mathbf{v}}\| \leq
\frac{\|L\|}{\delta}$.
\end{lem}
\begin{proof}
If $(f_{j})_{j \in I} \in
\mathcal{K}_{\mathbf{W}}$, then

\centerline{$\sum_{i \in
I}\|\frac{1}{v_{i}}LM_{i}(f_{j})_{j \in I}\|^{2}\leq
\frac{\|L\|^{2}}{\delta^{2}}\sum_{i \in I}\|M_{i}(f_{j})_{j \in
I}\|^{2}=\frac{\|L\|^{2}}{\delta^{2}}\|(f_{i})_{i \in I}\|^{2}.$}
\end{proof}

\begin{lem}\label{L gamma TVv Mk}
Let $({\bf V},{\bf v})$ be a fusion Bessel sequence for
$\mathcal{V}$. Then $\gamma(T_{\mathbf{V},{\bf v}} M_k)=v_{k}$ for
each $k \in I$.
\end{lem}
\begin{proof}
Note that ${\rm Ker}(T_{\mathbf{V},{\bf v}} M_k)=\{(g_{i})_{i \in
I}\in \mathcal{K}_{\mathbf{V}}: g_{k}=0\}$ and

\centerline{${\rm Ker}(T_{\mathbf{V},{\bf v}}
M_k)^{\perp}=\{(g_{i})_{i \in I}\in \mathcal{K}_{\mathbf{V}}:
g_{i}=0 \text{ if } i \neq k\}$.}

\noindent From this

$\gamma(T_{\mathbf{V},{\bf v}} M_k)=\text{inf}\{\|T_{\mathbf{V},{\bf
v}} M_k(g_{i})_{i \in I}\|:(g_{i})_{i \in I}\in
\mathcal{K}_{\mathbf{V}}, \|(g_{i})_{i \in I}\|=1, (g_{i})_{i \in I}
\in {\rm Ker}(T_{\mathbf{V},{\bf
v}} M_k)^{\perp}\}=\text{inf}\{\|v_{k}g_{k}\|:g_{k} \in V_{k},
\|g_{k}\|=1\}=v_{k}$.
\end{proof}
The following proposition gives an alternative for Lemma~\ref{L V Bessel} for a particular operator $L$.
\begin{prop}\label{P A invertible AV Bessel}
Let $(\mathbf{V},{\bf v})$ be a fusion Bessel sequence
for $\mathcal{V}$. Let $A \in
B(\mathcal{V})$, $\widetilde{\mathbf{v}}$ be a family of weights, and $\delta > 0$. If $A$ is invertible and $\delta \leq v_{i}/\widetilde{v}_{i}$ for all $i \in I$, then $(A\mathbf{V},\widetilde{\mathbf{v}})$ is a fusion Bessel
sequence for $\mathcal{V}$ with upper Bessel bound $\frac{\|AT_{\mathbf{V},{\bf v}}\|^{2}\|A\|^{2}}{\delta^{2}}$.
\end{prop}
\begin{proof}
By \cite[Remark 2.4 (4)]{Ruiz-Stojanoff (2008)} and Lemma~\ref{L gamma TVv Mk},

\centerline{$\frac{(\gamma(AT_{\mathbf{V},{\bf v}}M_i))^{2}}{\widetilde{v}_{i}^{2}} \geq \frac{\|A\|^{-2}(\gamma(T_{\mathbf{V},{\bf v}}
M_i))^{2}}{\widetilde{v}_{i}^{2}}=\frac{\|A\|^{-2}
v_{i}^{2}}{\widetilde{v}_{i}^{2}}\geq \delta^{2}\|A\|^{-2}$.}

\noindent So, the conclusion follows from Lemma~\ref{L V Bessel}.
\end{proof}

\subsection{Uniqueness and approximate consistent reconstruction}

Assume that $(\mathbf{W},{\bf w})$ is a fusion frame for
$\mathcal{W}$. Consider the samples $T_{\mathbf{W},{\bf
w}}^{\ast}f=(w_{i}\pi_{W_{i}}f)_{i \in I}$ of an unknown data vector $f \in
\mathcal{H}$. Here the aim is to reconstruct $f$ from
these samples using a fusion frame $({\bf V},{\bf v})$ for
$\mathcal{V}$ so that the reconstruction $f_{r} \in \mathcal{V}$ is
a good approximation of $f$. More precisely, for $\e \geq 0$ we ask for the following two
conditions:
\begin{enumerate}
  \item[(i)] \textit{Uniqueness of the reconstructed signal}: If $f, g \in \mathcal{V}$ and $T_{\mathbf{W},{\bf
w}}^{\ast}f=T_{\mathbf{W},{\bf w}}^{\ast}g$, then $f=g$.

  \item[(ii)] $\epsilon$-\textit{consistent reconstruction}: $||T_{\mathbf{W},{\bf w}}^{\ast}f_{r}-T_{\mathbf{W},{\bf w}}^{\ast}f|| \leq \epsilon ||f||$
for all $f \in \mathcal{H}$.

\end{enumerate}
Condition (i) is equivalent to
$\mathcal{V}\cap\mathcal{W}^{\perp}=\{0\}$ (see \cite{Heineken-Morillas (2018)}). If requirement (ii) is satisfied, we call $ f_{r}$  an $\epsilon$-\emph{consistent
reconstruction} of $ f $ in $ \Vsb $. Note that a $ 0 $-consistent
reconstruction is a \emph{consistent reconstruction} as defined in \cite{Heineken-Morillas (2018)}.
\begin{rem}
Under the hypothesis of Theorem~\ref{T dual fusion frame systems}, we observe that if $||T_{\mathbf{W},{\bf w}}^{\ast}f_{r}-T_{\mathbf{W},{\bf w}}^{\ast}f||
\leq \epsilon ||f||$ for all $f \in \mathcal{H}$ then $||T_{{\bf
w}\mathcal{F}}^{\ast}-T_{{\bf w}\mathcal{F}}^{\ast}f|| \leq \epsilon
\widetilde{\beta} ||f||$ for all $f \in \mathcal{H}$. Also, if $||T_{{\bf w}\mathcal{F}}^{\ast}f_r-T_{{\bf w}\mathcal{F}}^{\ast}f|| \leq
\epsilon ||f||$ for all $f \in \mathcal{H}$ then
$||T_{\mathbf{W},{\bf w}}^{\ast}f_{r}-T_{\mathbf{W},{\bf w}}^{\ast}f||
\leq \epsilon \beta ||f||$ for all $f \in \mathcal{H}$.
\end{rem}
Using an oblique dual fusion frame of $(\mathbf{W},{\bf w})$
in $ \Vsb $, we obtain the next bound for $\norm{ f_{r} -
\pi_{ \Vsb \Wsb^{ \perp } } f }$. The proof is similar to the frame
case (see \cite{Diaz-Heineken-Morillas (2023)}).
\begin{prop}\label{P cota reconstr consist dual G}
Let $\mathcal{H}=\mathcal{V}\oplus\mathcal{W}^{\perp}$. Let $(\mathbf{W},{\bf w})$ be a fusion frame for $ \Wsb $ and $(\mathbf{V},{\bf v})$ be an oblique $Q$-dual fusion frame of
$(\mathbf{W},{\bf w})$ in $ \Vsb $. Let $\e \geq 0$. If $ f \in \Hil$ and $ f_{r} $ is an $
\e $-consistent reconstruction of $f$ in $ \Vsb $, then $\norm{ f_{r} - \pi_{ \Vsb \Wsb^{ \perp } } f } \leq \e
\norm{ T_{\mathbf{V},{\bf
v}} Q}\norm{ f }$.
\end{prop}
If $(\mathbf{V},{\bf v})$ is the canonical oblique dual fusion frame of $(\mathbf{W},{\bf w})$ then $T_{\mathbf{V},{\bf
v}} Q=\pi_{\mathcal{V},\mathcal{W}^{\perp}}S_{\mathbf{W},{\bf
w}}^{\dagger}T_{\mathbf{W},{\bf w}}$. We also have, $T_{\mathbf{W},{\bf
w}}^{\dagger}=T_{\mathbf{W},{\bf
w}}^{\ast}S_{\mathbf{W},{\bf
w}}^{\dagger}$. So part (i) of the following result follows from Proposition~\ref{P cota reconstr
consist dual G}. Part (ii) is straightforward.
\begin{thm}\label{T consistent reconstruction}
Let $\mathcal{H}=\mathcal{V}\oplus\mathcal{W}^{\perp}$. Let $(\mathbf{W},{\bf w})$ be a fusion frame for $ \Wsb $, $ f \in \Hil$ and $ \e
\geq 0 $. The following holds:
\begin{enumerate}
\item[(i)] If $ f_{r} $ is an  $ \e $-consistent reconstruction of
$f$ in $ \Vsb $, then $\norm{ f_{r} \!-\! \pi_{ \Vsb \Wsb^{ \perp } } f
} \leq \e\!\norm{ \pi_{ \Vsb \Wsb^{ \perp } }\!}\!\norm{
T_{\mathbf{W},{\bf w}}^{ \dagger }}\!\norm{ f }$.
\item[(ii)] If $\norm{ f_{r} - \pi_{ \Vsb \Wsb^{ \perp } } f } \leq \frac{
\e }{ \norm{ T_{\mathbf{W},{\bf
w}} } } \norm{ f }$, then $ f_{r} $ is
an $ \e $-consistent reconstruction of $f$ in $ \Vsb $.
\end{enumerate}
\end{thm}
From the previous theorem we get the next
corollary that links $\e$-consistent
reconstruction with approximate oblique duality.
\begin{cor}
Let $\mathcal{H}=\mathcal{V}\oplus\mathcal{W}^{\perp}$. Let $(\mathbf{W},{\bf w})$ be a fusion frame for $ \Wsb $,
$(\mathbf{V},{\bf v})$ be a fusion frame for $ \Vsb $, $Q \in
B(\mathcal{K}_{\mathbf{W}}, \mathcal{K}_{\mathbf{V}})$ and $\e \geq
0 $. Then
\begin{enumerate}
\item[(i)]
If $f_{r}  = T_{{\bf V},{\bf v}}QT^{\ast}_{\mathbf{W},{\bf
w}} f $ is an
$ \frac{ \e }{ || \pi_{ \Vsb \Wsb^{ \perp } } || ||T_{\mathbf{W},{\bf
w}}^{ \dagger} || } $-consistent reconstruction of $f$
 in $ \Vsb $, for each $f \in \mathcal{H}$, then $(\mathbf{W},{\bf w})$ and $(\mathbf{V},{\bf v})$ are $ \e $-approximate oblique $Q$-dual
fusion frames.

\item[(ii)]
If $(\mathbf{V},{\bf v})$ is an $\e$-approximate oblique $Q$-dual fusion frame of $(\mathbf{W},{\bf w})$ in $\mathcal{V}$, then $ f_{r}  = T_{{\bf V},{\bf v}}Q T^{\ast}_{{\bf W},{\bf w}} f $ is an $ \e || T_{{\bf W},{\bf w}} ||
$-consistent reconstruction of $f$ in $ \Vsb $, for each $f \in
\mathcal{H}$. \qedhere
\end{enumerate}
\end{cor}

\section{Approximate oblique dual families}

In \cite{Heineken-Morillas-Benavente-Zakowicz (2014),
Heineken-Morillas (2014), Heineken-Morillas (2018)} it is shown that there is a link between (oblique) component preserving dual
fusion frames and the left inverses of the analysis
operator. We will show that similar results hold for $ \e $-approximate oblique component preserving dual
fusion frames.

\begin{lem}\label{L V,v dual fusion frame entonces Vi=ApiWj}
Let $\mathcal{H}=\mathcal{V}\oplus\mathcal{W}^{\perp}$. Let
$(\mathbf{W},{\bf w})$ be a fusion frame for $\mathcal{W}$ and $\e \geq 0.$ If
$({\bf V},{\bf v})$ is an $\epsilon$-approximate oblique component
preserving $Q$-dual fusion frame of $(\mathbf{W},{\bf w})$ on
$\mathcal{V}$ then $V_{i}=LM_{i}\mathcal{K}_{\mathbf{W}}$ for each
$i\in I$ for $L=T_{\mathbf{V},{\bf v}}Q \in
\mathfrak{L}_{T^{\ast}_{\mathbf{W},{\bf w}},\epsilon}^{\mathcal{V},\mathcal{W}^{\perp}}$.
\end{lem}
\begin{proof}
Let $Q\in B(\mathcal{K}_{\mathbf{W}},\mathcal{K}_{\mathbf{V}})$ be
component preserving such that $||T_{\mathbf{V},{\bf
v}}QT^{\ast}_{\mathbf{W},{\bf
w}}-\pi_{\mathcal{V},\mathcal{W}^{\perp}}|| \leq \epsilon$ and $L=T_{\mathbf{V},{\bf v}}Q$. Then
$L\in\mathfrak{L}_{T^{\ast}_{\mathbf{W},{\bf w}},\epsilon}^{\mathcal{V},\mathcal{W}^{\perp}}$. Since $Q$ is component
preserving, $LM_{i}\mathcal{K}_{\mathbf{W}}=T_{\mathbf{V},{\bf
v}}QM_{i}\mathcal{K}_{\mathbf{W}}=V_{i}$ for each $i \in
I$.\end{proof}
Reciprocally:
\begin{lem}\label{L Vi=ApiWj entonces V,v dual fusion frame}
Let $\mathcal{H}=\mathcal{V}\oplus\mathcal{W}^{\perp}$ and $\e \geq 0.$ Let
$(\mathbf{W},\mathbf{w})$ be a fusion frame for $\mathcal{W}$ and $L
\in \mathfrak{L}_{T^{\ast}_{\mathbf{W},{\bf w}},\epsilon}^{\mathcal{V},\mathcal{W}^{\perp}}$ such that
$V_{i}=LM_{i}\mathcal{K}_{\mathbf{W}}$ for each $i \in I$. Assume that $(\mathbf{V},\mathbf{v})$ is a fusion Bessel sequence for
$\mathcal{V}$ and $Q_{L,{\bf v}}$ defined in (\ref{E QA}) is a well defined bounded operator. Then
$(\mathbf{V},\mathbf{v})$ is an $\epsilon$-approximate
 oblique component preserving $Q_{L,\mathbf{v}}$-dual fusion frame of
$(\mathbf{W},\mathbf{w})$ on $\mathcal{V}$.
\end{lem}

\begin{proof}
We can conclude from the hypotheses that $\{\sum_{i \in I}g_{i}: g_{i} \in V_{i} \}={\rm
Im}(L)=\mathcal{V}$, $Q_{L,{\bf v}}$ is component preserving and
$L=T_{\mathbf{V},{\bf v}}Q_{L,{\bf v}}$. Since
$L\in\mathfrak{L}_{T^{\ast}_{\mathbf{W},{\bf w}},\epsilon}^{\mathcal{V},\mathcal{W}^{\perp}}$, $||T_{\mathbf{V},{\bf v}}Q_{L,{\bf v}}T^{\ast}_{\mathbf{W},{\bf
w}}-\pi_{\mathcal{V},\mathcal{W}^{\perp}}|| \leq \epsilon$. Thus
$(\mathbf{V},{\bf v})$ is an $\epsilon$-approximate
 oblique component preserving $Q_{L,{\bf v}}$-dual fusion frame of
$(\mathbf{W},{\bf w})$ on $\mathcal{V}$.
\end{proof}
In Lemmas~\ref{L V Bessel} and \ref{L Q acotada} one can find sufficient conditions for $(\mathcal{V},v)$ being a fusion Bessel sequence for $\mathcal{V}$ and for $Q_{L,v}$ being a well
defined bounded operator in Lemma~\ref{L Vi=ApiWj entonces V,v dual
fusion frame}. For the case in which $\mathcal{W}$ and $\mathcal{V}$
are finite-dimensional, Lemma~\ref{L V,v dual fusion frame entonces
Vi=ApiWj} and Lemma~\ref{L Vi=ApiWj entonces V,v dual fusion frame}
lead to the following characterization:
\begin{thm}\label{T V,v dual fusion
frame sii Vi=ApiWj} Let $\mathcal{W}$ and $\mathcal{V}$ be
finite-dimensional subspaces of the Hilbert space $\mathcal{H}$ such
that $\mathcal{H}=\mathcal{V}\oplus \mathcal{W}^{\perp}$ and $\epsilon \geq 0$. Let
$(\mathbf{W},{\bf w})$ be a fusion frame for $\mathcal{W}$. Then
$({\bf V},{\bf v})$ is an $\epsilon$-approximate oblique component preserving $Q$-dual
fusion frame of $(\mathbf{W},{\bf w})$ on $\mathcal{V}$ if and only
if $V_{i}=LM_{i}\mathcal{K}_{\mathbf{W}}$ for each $i\in I$ and
$Q=Q_{L,{\bf v}},$ for some $L\in\mathfrak{L}_{T^{\ast}_{\mathbf{W},{\bf w}},
\epsilon}^{\mathcal{V},\mathcal{W}^{\perp}}$ where ${\rm
Im}(L)=\mathcal{V}$. Furthermore, any element of $L\in\mathfrak{L}_{T^{\ast}_{\mathbf{W},{\bf w}},
\epsilon}^{\mathcal{V},\mathcal{W}^{\perp}}$ with ${\rm
Im}(L)=\mathcal{V}$ is of the form $T_{{\bf V},{\bf v}}Q$ where
$({\bf V},{\bf v})$ is some $\epsilon$-approximate oblique component preserving $Q$-dual
fusion frame of $(\mathbf{W},{\bf w})$ on $\mathcal{V}$.
\end{thm}
\begin{rem}
\cite[Proposition 2.4]{Heineken-Morillas (2018)},
Theorem~\ref{T V,v dual fusion frame sii Vi=ApiWj} and \cite[Theorem
3.5]{Heineken-Morillas (2014)} tell us that if $\mathcal{W}$ and
$\mathcal{V}$ are finite-dimensional, there is a bijection between $\epsilon$-approximate component preserving dual and $\epsilon$-approximate  oblique component preserving dual fusion
frames.
\end{rem}
\begin{rem}
We can conclude from Theorem~\ref{T V,v dual fusion frame sii
Vi=ApiWj} that if $\mathcal{W}$ and $\mathcal{V}$ are finite-dimensional
we can always associate to any $\epsilon$-approximate oblique $Q$-dual fusion frame $({\bf
V},{\bf v})$ of $(\mathbf{W},{\bf w})$ the $\epsilon$-approximate oblique component preserving $Q_{L,{\bf
\tilde{v}}}$-dual fusion frame
$\{(LM_{i}\mathcal{K}_{\mathbf{W}},\tilde{v}_{i})\}_{i \in I}$ where
$L=T_{{\bf V},{\bf v}}Q$ and $\{\widetilde{v}_{i}\}_{i \in I}$ are
arbitrary weights. Moreover, if $Q$ is block diagonal, then
$Q_{T_{{\bf V},{\bf v}}Q,{\bf v}}(f_i)_{i \in I}=Q(f_i)_{i \in I}$
for each $(f_i)_{i \in I} \in \mathcal{K}_{\mathbf{W}}$.
\end{rem}

\begin{rem}\label{R conjunto inversas oblicuas de TWw*}
It can be seen that $L \in \mathfrak{L}_{T^{\ast}_{\mathbf{W},{\bf w}},\epsilon}^{\mathcal{V},\mathcal{W}^{\perp}}$ if and only if 

\centerline{$L=\pi_{\mathcal{V},\mathcal{W}^{\perp}}S_{\mathbf{W},{\bf
w}}^{\dagger}T_{\mathbf{W},{\bf
w}}+B(I_{\mathcal{K}_{\mathbf{W}}}-T^{\ast}_{\mathbf{W},{\bf
w}}S_{\mathbf{W},{\bf w}}^{\dagger}T_{\mathbf{W},{\bf w}})+R$}

\noindent for some $B, R\in B(\mathcal{K}_{\mathbf{W}},\mathcal{V})$ such that $||RT^{\ast}_{\mathbf{W},{\bf
w}}|| \leq
\epsilon$. Alternatively, $L \in \mathfrak{L}_{T^{\ast}_{\mathbf{W},{\bf w}},\epsilon}^{\mathcal{V},\mathcal{W}^{\perp}}$ if and only if

\centerline{$L=\pi_{\mathcal{V},\mathcal{W}^{\perp}}S_{\mathbf{W},{\bf
w}}^{\dagger}T_{\mathbf{W},{\bf w}}+B$}

\noindent for some $B\in
B(\mathcal{K}_{\mathbf{W}},\mathcal{V})$ such that $||BT^{\ast}_{\mathbf{W},{\bf w}}|| \leq \epsilon$.
\end{rem}
Now, we analyze approximate oblique component preserving
dual fusion frame systems in relation with a given fusion frame for a closed
subspace of $\mathcal{H}$, local approximate dual frames and an approximate oblique left
inverse of its analysis operator.
\begin{thm}\label{T L TW T fusion frame system dual 3}
Let $\mathcal{H}=\mathcal{V} \oplus
\mathcal{W}^{\perp}$ and $\epsilon \geq 0$. Assume
that $(\mathbf{W}, {\bf w})$ is a fusion frame for
$\mathcal{W}$, $L\in \mathfrak{L}_{T^{\ast}_{\mathbf{W},{\bf w}},\epsilon}^{\mathcal{V},\mathcal{W}^{\perp}}$ and ${\bf v}$ is a collection of weights such that $\inf_{i \in I}v_i>0$. For
each $i\in I$ let $\{f_{i,l}\}_{l\in L_i}$ and
$\{\tilde{f}_{i,l}\}_{l\in L_i}$ be $\epsilon_{i}$-approximate
dual frames for $W_i$, $\paren{w_{i}\epsilon_{i}}_{i \in I} \in
\ell^{2}\paren{I}$, $\beta_i$ upper frame bounds of
$\{f_{i,l}\}_{l\in L_i}$ such that $\sup_{i \in I}\beta_i<\infty$,
$\tilde{\alpha}_i$ and $\tilde{\beta}_i$ frame bounds of
$\{\tilde{f}_{i,l}\}_{l\in L_i}$ such that $\sup_{i \in
I}\widetilde{\beta}_i<\infty$, $\mathcal{G}_i=\{
\frac{1}{v_{i}}L(\chi_{i}(j)\tilde{f}_{i,l})_{j \in I}\}_{l\in L_i}$
and $V_{i}=\overline{{\rm span}}\,\mathcal{G}_i$. Then
\begin{enumerate}
 \item  $\mathcal{G}_i$ is a frame for $V_{i}$ with frame
bounds $\|L^{\dagger}\|^{-2}\frac{\widetilde{\alpha}_i}{v_{i}^{2}}$
and $\|L\|^{2}\frac{\widetilde{\beta}_i}{v_{i}^{2}}.$
 \item $(\mathbf{V},\mathbf{v},\mathcal{G})$ is an $\paren{\norm{L}\paren{\sum_{i \in
I}w_{i}^{2}\epsilon_{i}^{2}}^{1/2}+\epsilon}$-approximate oblique
component preserving dual fusion frame system of $(\mathbf{W}, {\bf
w},
 \mathcal{F})$ on $\mathcal{V}$.
  \item If $\epsilon_{i}=0$ for each $i \in I$, then
$Q_{L,\mathbf{v}}=C_{\mathcal{G}}C_{\mathcal{F}}^{\ast}$.
\end{enumerate}
\end{thm}
\begin{proof}
(1) holds by \cite[Proposition 5.3.1]{Christensen (2016)}.

Using Theorem~\ref{T wF marco sii WwF fusion frame system}, as in the
proof of \cite[Proposition 6.8]{Heineken-Morillas (2018)}, we obtain
that $(\mathbf{V},\mathbf{v},\mathcal{G})$ is a fusion Bessel system
for $\mathcal{H}$ with upper frame bounds $\|L\|^{2}\sup_{i \in
I}\widetilde{\beta}_i$. We also have that $Q_{L,\mathbf{v}}$ is a
well defined bounded operator with
$||Q_{L,\mathbf{v}}||\leq\frac{||L||^{2}}{\inf_{i \in I}v_{i}^{2}}$.

Let $f \in \mathcal{H}$. We have,
\begin{align*}
T_{\mathbf{V},{\bf
v}}C_{\mathcal{G}}C_{\mathcal{F}}^{\ast}T^{\ast}_{\mathbf{W},{\bf
w}}f&=\sum_{i \in
I}w_{i}v_{i}T_{\mathcal{G}_{i}}T_{\mathcal{F}_{i}}^{\ast}\pi_{\mathcal{W}_{i}}f\\
&=\sum_{i \in I}w_{i}\sum_{l\in
L_i}\langle\pi_{\mathcal{W}_{i}}f,f_{i,l}\rangle
L(\chi_{i}(j)\tilde{f}_{i,l})_{j \in I}=L\paren{w_{i}\sum_{l\in
L_i}\langle\pi_{\mathcal{W}_{i}}f,f_{i,l}\rangle \tilde{f}_{i,l}}_{i
\in I}
\end{align*}
and $T_{\mathbf{V},{\bf v}}Q_{L,{\bf v}}T^{\ast}_{\mathbf{W},{\bf
w}}f=L\paren{w_{i}\pi_{\mathcal{W}_{i}}f}_{i \in I}$.
Then
\begin{align*}
\norm{T_{\mathbf{V},{\bf
v}}C_{\mathcal{G}}C_{\mathcal{F}}^{\ast}T^{\ast}_{\mathbf{W},{\bf
w}}f-T_{\mathbf{V},{\bf v}}Q_{L,{\bf v}}T^{\ast}_{\mathbf{W},{\bf w}}f}
&\leq\norm{L}\paren{\sum_{i \in I}w_{i}^{2}\norm{\sum_{l\in
L_i}\langle\pi_{\mathcal{W}_{i}}f,f_{i,l}\rangle
\tilde{f}_{i,l}-\pi_{\mathcal{W}_{i}}f}^{2}}^{1/2}\\
&\leq\norm{L}\!\paren{\sum_{i \in
I}w_{i}^{2}\epsilon_{i}^{2}\norm{\pi_{\mathcal{W}_{i}}f}^{2}}^{1/2}\!\!\leq
\norm{L}\paren{\sum_{i \in
I}w_{i}^{2}\epsilon_{i}^{2}}^{1/2}\!\!\!\!\norm{f}.
\end{align*}

By Lemma~\ref{L Vi=ApiWj entonces V,v dual fusion frame},
\begin{align*}
\norm{T_{\mathbf{V},{\bf
v}}C_{\mathcal{G}}C_{\mathcal{F}}^{\ast}T^{\ast}_{\mathbf{W},{\bf
w}}f-\pi_{\mathcal{V},\mathcal{W}^{\perp}}f}\leq&\norm{T_{\mathbf{V},{\bf
v}}C_{\mathcal{G}}C_{\mathcal{F}}^{\ast}T^{\ast}_{\mathbf{W},{\bf
w}}f-T_{\mathbf{V},{\bf v}}Q_{L,{\bf v}}T^{\ast}_{\mathbf{W},{\bf
w}}f}\\&+\norm{T_{\mathbf{V},{\bf v}}Q_{L,{\bf
v}}T^{\ast}_{\mathbf{W},{\bf
w}}f-\pi_{\mathcal{V},\mathcal{W}^{\perp}}f}\\
\leq& \norm{L}\paren{\sum_{i \in
I}w_{i}^{2}\epsilon_{i}^{2}}^{1/2}\norm{f}+\epsilon\norm{f}.
\end{align*}
This proves (2).

Finally, (3) can be proved as \cite[Proposition
6.8(3)]{Heineken-Morillas (2018)}.
\end{proof}
In the next theorem we consider approximate oblique dual
fusion frame systems in relation to a given frame for a subspace and an approximate oblique left inverse of its analysis operator.
\begin{thm}\label{T L TF T fusion frame system dual}
Let $\mathcal{H}=\mathcal{V} \oplus
\mathcal{W}^{\perp} $ and $\epsilon \geq 0$.  Let ${\bf w}$ and ${\bf v}$ be two
collections of weights such that $\inf_{i \in I}v_i>0$. Let ${\bf
w}\mathcal{F}$ be a frame for $\mathcal{W}$ with local upper frame
bounds $\beta_i$ such that $\sup_{i \in I}\beta_i<\infty$, $L \in
\mathfrak{L}_{T_{{\bf w}\mathcal{F}}^{\ast},\epsilon}^{\mathcal{V},\mathcal{W}^{\perp}}$ and
$\{\{e_{i,l}\}_{l\in L_i}\}_{i\in I}$ be the standard basis for
$\oplus_{i \in I}\ell^2(L_{i})$. For each $i \in I$, set
$W_{i}=\overline{\text{span}}\{f_{i,l}\}_{l\in L_i}$ and
$V_{i}=\overline{\text{span}}\{\frac{1}{v_{i}}Le_{i,l}\}_{l\in
L_i}$. Let $\mathcal{G}=\{\{\frac{1}{v_{i}}Le_{i,l}\}_{l\in
L_i}\}_{i \in I}.$ Then
\begin{enumerate}
 \item $\{\frac{1}{v_{i}}Le_{i,l}\}_{l\in L_i}$ is a frame for $V_{i}$ with frame
bounds $\frac{\|L^{\dagger}\|^{-2}}{v_{i}^{2}}$ and
$\frac{\|L\|^{2}}{v_{i}^{2}}$.
 \item $(\mathbf{V}, {\bf v}, \mathcal{G})$ is an $\epsilon$-approximate oblique dual fusion frame system of $(\mathbf{W}, {\bf w},
 \mathcal{F})$ on $\mathcal{V}$.
\end{enumerate}
\end{thm}
\begin{proof}
The proof is similar to the one of \cite[Proposition 6.9]{Heineken-Morillas (2018)}. Here, we use \cite[Proposition 5.3.1]{Christensen
(2016)}, \cite[Lemma 4.5]{Diaz-Heineken-Morillas (2023)}, and Theorem~\ref{T dual
fusion frame systems}
\end{proof}

%%%%%%%%%%%%%%%%%%%%%%%%%%%%%%%%%%%%%%%%
\section{Better approximation}

Assume that $\mathcal{H}=\mathcal{V}\oplus\mathcal{W}^{\perp}$. The next theorem shows that if $(\mathbf{W},{\bf w})$ and $ (\mathbf{V},{\bf v})$ are $\e$-approximate oblique $Q$-dual fusion frames with $0 \leq \e < 1$, then we can construct
from $ (\mathbf{V},{\bf v})$ other approximate oblique dual fusion frames
in $ \Vsb $ of $(\mathbf{W},{\bf w})$ that reconstruct as ``good" as
we want. This is similar to \cite[Proposition
5.1]{Diaz-Heineken-Morillas (2023)} for approximate oblique dual
frames (see also \cite[Proposition 3.2]{Christensen-Laugesen (2010)}
for approximate dual frames). 

We define $A=\sum_{ n=0}^{ \infty } ( \pi_{ \Vsb \Wsb^{ \bot } } -
T_{\mathbf{V},{\bf v}} Q T_{\mathbf{W},{\bf w}}^{ \ast } )^n$. By Neumann's Theorem, $I
- ( \pi_{ \Vsb \Wsb^{ \bot } } - T_{\mathbf{V},{\bf v}} Q
T_{\mathbf{W},{\bf w}}^{ \ast } ) $ is invertible and $A=( I - ( \pi_{ \Vsb \Wsb^{ \bot } } -
T_{\mathbf{V},{\bf v}} Q T_{\mathbf{W},{\bf w}}^{ \ast } ) )^{ -1 }$. We set $A_{N}=\sum_{ n=0}^{ N } ( \pi_{ \Vsb \Wsb^{ \bot } } -
T_{\mathbf{V},{\bf v}} Q T_{\mathbf{W},{\bf w}}^{ \ast } )^n$.
\begin{thm}\label{T better approximation}
Let $\mathcal{H}=\mathcal{V}\oplus\mathcal{W}^{\perp}$. Let $(\mathbf{W},{\bf w})$ be a fusion frame for $ \Wsb $ and
$ (\mathbf{V},{\bf v})$ an $\e$-approximate oblique $Q$-dual fusion frame in
$\Vsb $ with $0 \leq \e < 1$. Let $\widetilde{{\bf w}}=\{\widetilde{w}_{k}\}_{k \in I}$ and $\widetilde{{\bf v}}=\{\widetilde{v}_{k}\}_{k \in I}$ be families of weights. If $(A_{N}\mathbf{V},\widetilde{{\bf v}})$ is
a fusion Bessel sequence for $\mathcal{V}$ and
$Q_{A_{N}T_{\mathbf{V},{\bf v}},\widetilde{\mathbf{v}}}$
is a well defined bounded operator, then
$(A_{N}\mathbf{V},{\bf v})$ is an
$\e^{N+1}$-approximate oblique $Q_{A_{N}T_{\mathbf{V},{\bf v}},\widetilde{\mathbf{v}}}Q$-dual fusion frame of
$(\mathbf{W},{\bf w})$ in $ \Vsb $ and $\norm{ \pi_{ \Vsb \Wsb^{ \bot } } -
T_{A_{N}\mathbf{V},\widetilde{{\bf v}}}
Q_{A_{N}T_{\mathbf{V},{\bf v}},\widetilde{\mathbf{v}}}Q
T_{\mathbf{W},{\bf w}}^{ \ast } } \leq \e^{N+1} \longrightarrow 0$ as
$N\rightarrow \infty$.
\end{thm}
\begin{proof}
Let $ f\in\Hil $. Then
\begin{align*}
T_{A_{N}\mathbf{V},\widetilde{\bf v}}
Q_{A_{N}T_{\mathbf{V},{\bf v}},\widetilde{\mathbf{v}}}Q
T_{\mathbf{W},{\bf w}}^{ \ast } f &= \sum_{ n=0 }^{ N } ( \pi_{
\Vsb \Wsb^{ \bot } } - T_{\mathbf{V},{\bf v}} Q T_{\mathbf{W},{\bf
w}}^{ \ast } )^n T_{\mathbf{V},{\bf v}} Q T_{\mathbf{W},{\bf w}}^{
\ast } f\\
&= \sum_{ n=0 }^{ N } ( \pi_{ \Vsb \Wsb^{ \bot } } -
T_{\mathbf{V},{\bf v}} Q T_{\mathbf{W},{\bf w}}^{ \ast } )^n ( \pi_{
\Vsb \Wsb^{ \bot } } - ( \pi_{ \Vsb \Wsb^{ \bot } } -
T_{\mathbf{V},{\bf v}} Q T_{\mathbf{W},{\bf w}}^{ \ast } ) ) \pi_{ \Vsb \Wsb^{ \bot } } f\\
&= \pi_{ \Vsb \Wsb^{ \bot } } f - ( \pi_{ \Vsb \Wsb^{ \bot } } -
T_{\mathbf{V},{\bf v}} Q T_{\mathbf{W},{\bf w}}^{ \ast } )^{ N+1 }
\pi_{ \Vsb \Wsb^{ \bot } } f.
\end{align*}
Since $ \Wsb^{ \bot } \subseteq \mathcal{N}( \pi_{ \Vsb \Wsb^{ \bot
} } - T_{\mathbf{V},{\bf v}} Q T_{\mathbf{W},{\bf w}}^{ \ast } ) $,
this implies that 

\centerline{$\pi_{ \Vsb \Wsb^{ \bot } }f -
T_{A_{N}\mathbf{V},\widetilde{\bf v}}
Q_{A_{N}T_{\mathbf{V},{\bf v}},\widetilde{\mathbf{v}}}Q
T_{\mathbf{W},{\bf w}}^{ \ast }f = ( \pi_{ \Vsb \Wsb^{ \bot } }
-T_{\mathbf{V},{\bf v}} Q T_{\mathbf{W},{\bf w}}^{ \ast } )^{ N+1
}f$.} 

\noindent So

\centerline{$\norm{\pi_{ \Vsb \Wsb^{ \bot } } -
T_{A_{N}\mathbf{V},\widetilde{\bf v}}
Q_{A_{N}T_{\mathbf{V},{\bf v}},\widetilde{\mathbf{v}}}Q
T_{\mathbf{W},{\bf w}}^{ \ast }} \leq \norm{ ( \pi_{ \Vsb \Wsb^{
\bot } } -T_{\mathbf{V},{\bf v}} Q T_{\mathbf{W},{\bf w}}^{ \ast }
)^{ N+1 }} \leq\e^{N+1} \rightarrow 0$ as $N\rightarrow \infty$.}
\end{proof}
\begin{rem}
Observe that from the preceding proof we have

\centerline{$\pi_{ \Vsb \Wsb^{ \bot } } -
T_{A_{N}\mathbf{V},\widetilde{\bf v}}
Q_{A_{N}T_{\mathbf{V},{\bf v}},\widetilde{\mathbf{v}}}Q
T_{\mathbf{W},{\bf w}}^{ \ast } = ( \pi_{ \Vsb \Wsb^{ \bot } } -
T_{\mathbf{V},{\bf v}} Q T_{\mathbf{W},{\bf w}}^{ \ast } )^{ N+1
}$.}

\noindent Note that the operator $T_{A_{N}\mathbf{V},\widetilde{\bf
v}}Q_{A_{N}T_{\mathbf{V},{\bf v}},\widetilde{\mathbf{v}}}Q$
can be expressed as $T_{A_{N}\mathbf{V},\widetilde{\bf
v}}Q_{A_{N}T_{\mathbf{V},{\bf v}},\widetilde{\mathbf{v}}}Q
= A_{N} T_{\mathbf{V},{\bf v}} Q$. If $(A_{N}^{\ast}\mathbf{W},\widetilde{{\bf w}})$ is
a fusion Bessel sequence for $\mathcal{W}$ and
$Q_{A_{N}^{\ast}T_{\mathbf{W},{\bf w}},\widetilde{\mathbf{w}}}$
is a well defined bounded operator then
$T_{A_{N}^{\ast}\mathbf{W},\widetilde{\bf
w}}Q_{A_{N}^{\ast}T_{\mathbf{W},{\bf w}},\widetilde{\mathbf{w}}}Q^{\ast}
= A_{N}^{\ast} T_{\mathbf{W},{\bf w}} Q^{\ast}$,

\centerline{$\pi_{ \Wsb \Vsb^{ \bot } } -
T_{A_{N}^{\ast}\mathbf{W},\widetilde{\bf w}}
Q_{A_{N}^{\ast}T_{\mathbf{W},{\bf w}},\widetilde{\mathbf{w}}}Q^{\ast}
T_{\mathbf{V},{\bf v}}^{ \ast } = ( \pi_{ \Wsb \Vsb^{ \bot } } -
T_{\mathbf{W},{\bf w}} Q^{\ast} T_{\mathbf{V},{\bf v}}^{ \ast } )^{ N+1
}$}

\noindent and

\centerline{$\norm{ \pi_{ \Wsb \Vsb^{ \bot } } -
T_{A_{N}^{\ast}\mathbf{W},\widetilde{\bf w}}
Q_{A_{N}^{\ast}T_{\mathbf{W},{\bf w}},\widetilde{\mathbf{w}}}Q^{\ast}
T_{\mathbf{V},{\bf v}}^{ \ast } } \leq \norm{ \pi_{ \Wsb \Vsb^{ \bot
} } - T_{\mathbf{W},{\bf w}} Q^{\ast} T_{\mathbf{V},{\bf v}}^{ \ast }
}^{ N+1 } \leq\e^{N+1} \rightarrow 0$.}

From the previous results, we observe that we can improve
the approximation given by the approximate oblique dual fusion
frames $(\mathbf{W},{\bf w})$ and $(\mathbf{V},{\bf v})$, modifying
either of them.
\end{rem}
In the next theorem we show that the fusion frames
$(A_{N}^{\ast}\mathbf{W},\widetilde{{\bf w}})$ and
$(A_{N}\mathbf{V},\widetilde{{\bf v}})$ defined
before are linked to oblique dual fusion frames of $(\mathbf{W},{\bf
w})$ and $(\mathbf{V},{\bf v})$.

\begin{thm}
Let $\mathcal{H}=\mathcal{V}\oplus\mathcal{W}^{\perp}$. Let $(\mathbf{W},{\bf w})$ be a fusion frame for $\Wsb$ and $
(\mathbf{V},{\bf v})$ be a fusion frame for $\Vsb $. Assume that
$(\mathbf{W},{\bf w})$ and $ (\mathbf{V},{\bf v})$ are
$\e$-approximate oblique $Q$-dual fusion frames with $0 \leq \e <
1$. Let $\widetilde{{\bf v}}=\{\widetilde{v}_{k}\}_{k \in I}$ and
$\widetilde{{\bf w}}=\{\widetilde{w}_{k}\}_{k \in I}$ be two
families of weights. Assume that
$(A\mathbf{V},\widetilde{{\bf v}})$ and
$(A^{\ast}\mathbf{W},\widetilde{{\bf w}})$ are
fusion Bessel sequences for $\mathcal{V}$ and $\mathcal{W}$, respectively. Assume also that
$Q_{AT_{\mathbf{V},{\bf v}},\widetilde{\mathbf{v}}}$ and
$Q_{A^{\ast}T_{\mathbf{W},{\bf w}},\widetilde{\mathbf{w}}}$ are
well defined bounded operators. Then:
\begin{enumerate}
  \item[(i)] $(A\mathbf{V},\widetilde{{\bf v}})$ is an oblique
$Q_{AT_{\mathbf{V},{\bf v}},\widetilde{\mathbf{v}}}Q$-dual
fusion frame of $(\mathbf{W},{\bf w})$ in $\Vsb$.
  \item[(ii)] $(A^{\ast}\mathbf{W},\widetilde{{\bf w}})$ is an oblique
$Q_{A^{\ast}T_{\mathbf{W},{\bf w}},\widetilde{\mathbf{w}}}Q^{\ast}$-dual
fusion frame of $(\mathbf{V},{\bf v})$ in $\Wsb$.
  \item[(iii)] If $Q$ is surjective (injective) and $\widetilde{Q}: \mathcal{K}_{\mathbf{V}} \rightarrow
\mathcal{K}_{\mathbf{W}}$ is a bounded right or left inverse of $Q$, then $(A\mathbf{V},\widetilde{{\bf v}})$ is an
$\frac{\e}{1-\e}$-approximate oblique
$Q_{AT_{\mathbf{V},{\bf v}},\widetilde{\mathbf{v}}}QQ_{A^{\ast}T_{\mathbf{W},{\bf w}},\widetilde{\mathbf{w}}}^{\ast}$-dual frame of $(A^{\ast}\mathbf{W},\widetilde{{\bf
w}})$ in $\Vsb$.
\end{enumerate}
\end{thm}
\begin{proof}
Note that $T_{A_{N}\mathbf{V},\widetilde{\bf
v}}Q_{A_{N}T_{\mathbf{V},{\bf v}},\widetilde{\mathbf{v}}}Q
= A_N T_{\mathbf{V},{\bf v}} Q \rightarrow A T_{\mathbf{V},{\bf v}}
Q = T_{A\mathbf{V},\widetilde{\bf
v}}Q_{AT_{\mathbf{V},{\bf v}},\widetilde{\mathbf{v}}}Q$ as $
N \rightarrow \infty $. By Theorem~\ref{T better approximation},
\begin{align*}
||\pi_{\Vsb \Wsb^{\bot}}-T_{A\mathbf{V},\widetilde{\bf
v}}Q_{AT_{\mathbf{V},{\bf v}},\widetilde{\mathbf{v}}}QT_{\mathbf{W},{\bf w}}^{ \ast } || \leq & \,||\pi_{\Vsb
\Wsb^{\bot}}-T_{A_{N}\mathbf{V},\widetilde{\bf
v}}Q_{A_{N}T_{\mathbf{V},{\bf v}},\widetilde{\mathbf{v}}}Q
T_{\mathbf{W},{\bf w}}^{ \ast } ||\\
& +||T_{A_{N}\mathbf{V},\widetilde{\bf
v}}Q_{A_{N}T_{\mathbf{V},{\bf v}},\widetilde{\mathbf{v}}}Q
T_{\mathbf{W},{\bf w}}^{ \ast
}-T_{A\mathbf{V},\widetilde{\bf
v}}Q_{AT_{\mathbf{V},{\bf v}},\widetilde{\mathbf{v}}}Q T_{\mathbf{W},{\bf w}}^{ \ast }||\\
\leq & \,\e^{ N+1 } +||T_{A_{N}\mathbf{V},\widetilde{\bf
v}}Q_{A_{N}\mathbf{V},\widetilde{{\bf
v}}}-T_{A\mathbf{V},\widetilde{\bf
v}}Q_{AT_{\mathbf{V},{\bf v}},\widetilde{\mathbf{v}}}Q||\, ||
T_{\mathbf{W},{\bf w}}||
\end{align*}
The right hand side tends to $0$ as $ N \rightarrow \infty $,
 so we have $T_{A\mathbf{V},\widetilde{\bf
v}}Q_{AT_{\mathbf{V},{\bf v}},\widetilde{\mathbf{v}}}QT_{\mathbf{W},  {\bf w}}^{ \ast }=\pi_{\Vsb \Wsb^{\bot}}$. Therefore
$(A\mathbf{V},\widetilde{{\bf v}})$ is an
oblique $Q_{AT_{\mathbf{V},{\bf v}},\widetilde{\mathbf{v}}}Q$-dual fusion frame of $(\mathbf{W},{\bf w})$ in $\Vsb$.
Analogously
$(A^{\ast}\mathbf{W},\widetilde{{\bf w}})$ is an oblique $Q_{A^{\ast}T_{\mathbf{W},{\bf w}},\widetilde{\mathbf{w}}}Q^{\ast}$-dual
fusion frame of $(\mathbf{V},{\bf v})$ in $\Wsb$. Hence, (i) and (ii) hold.

From (i) and $\mathcal{R}(I-A)\subseteq\Vsb$,
\begin{align*}
\pi_{\Vsb \Wsb^{\bot}}-T_{A\mathbf{V},\widetilde{\bf
v}}Q_{AT_{\mathbf{V},{\bf v}},\widetilde{\mathbf{v}}}QQ_{A^{\ast}T_{\mathbf{W},{\bf w}},\widetilde{\mathbf{w}}}^{\ast}
T_{A^{\ast}\mathbf{W},\widetilde{\bf w}}^{ \ast }&=\pi_{\Vsb \Wsb^{\bot}}-A T_{\mathbf{V},{\bf v}} Q
T_{\mathbf{W},{\bf w}}^*A\\&=\pi_{\Vsb
\Wsb^{\bot}}-T_{A\mathbf{V},\widetilde{\bf
v}}Q_{AT_{\mathbf{V},{\bf v}},\widetilde{\mathbf{v}}}Q
T_{\mathbf{W},{\bf w}}^*A \\&= \pi_{\Vsb \Wsb^{\bot}}(I-A)=I-A.
\end{align*}

\noindent So $||\pi_{\Vsb \Wsb^{\bot}}-T_{A\mathbf{V},\widetilde{\bf
v}}Q_{AT_{\mathbf{V},{\bf v}},\widetilde{\mathbf{v}}}QQ_{A^{\ast}T_{\mathbf{W},{\bf w}},\widetilde{\mathbf{w}}}^{\ast}
T_{A^{\ast}\mathbf{W},\widetilde{\bf w}}^{ \ast }|| = \| \sum_{ n=1}^{ \infty } (\pi_{ \Vsb \Wsb^{ \bot } }-T_{\mathbf{V},{\bf v}} Q T_{\mathbf{W},{\bf w}}^{ \ast })^n \| \leq \frac{\e}{1-\e}$. This shows (iii).
\end{proof}
If we apply Proposition~\ref{P A invertible AV Bessel} (or Lemma~\ref{L V Bessel}) and Lemma~\ref{L Q acotada}, we get sufficient conditions on the weights $\widetilde{{\bf v}}$ and $\widetilde{{\bf w}}$ for $(A\mathbf{V},\widetilde{{\bf v}})$ and $(A^{\ast}\mathbf{w},\widetilde{{\bf w}})$ to be fusion Bessel sequences for $\mathcal{V}$ and $Q_{AT_{\mathbf{V},{\bf v}},\widetilde{\mathbf{v}}}$ and
$Q_{A^{\ast}T_{\mathbf{W},{\bf w}},\widetilde{\mathbf{w}}}$ to be bounded operators.
\begin{rem}\label{R Fmonio Gmonio no duales}
From the previous equalities, if $(A\mathbf{V},\widetilde{{\bf v}})$ is an oblique $Q_{AT_{\mathbf{V},{\bf v}},\widetilde{\mathbf{v}}}QQ_{A^{\ast}T_{\mathbf{W},{\bf w}},\widetilde{\mathbf{w}}}^{\ast}$-dual fusion frame
of $(A^{\ast}\mathbf{W},\widetilde{{\bf w}})$ in $\Vsb$, then
$A=I$ and $T_{\mathbf{V},{\bf v}} Q T_{\mathbf{W},{\bf w}}^{ \ast }=\pi_{ \Vsb \Wsb^{ \bot } }$, i.e., $(\mathbf{V},{\bf v})$ is an oblique $Q$-dual fusion frame of $(\mathbf{W},{\bf w})$ in $\Vsb$.
\end{rem}
From Theorem~\ref{T A invertible An AnV} below (or Lemma~\ref{L V Bessel}) and Lemma~\ref{L Q acotada} we obtain sufficient conditions on the weights $\widetilde{\mathbf{v}}$  to assure that $(A_{N}\mathbf{V},\widetilde{{\bf v}})$
is a fusion Bessel sequence for $\mathcal{V}$ and
$Q_{A_{N}T_{\mathbf{V},{\bf v}},\widetilde{\mathbf{v}}}$
is a well defined bounded operator. To prove Theorem~\ref{T A invertible An AnV} we need the following result.
\begin{lem}\label{L A B}
Let $A, B \in B(\mathcal{H},\mathcal{K})$ be such that ${\rm
Ker}(A) \subseteq {\rm Ker}(B)$. Then $\gamma(A)-\gamma(B)\leq ||B - A ||$.
\end{lem}
\begin{proof}
We have $|\, ||Bx|| - ||A x|| \,| \leq ||Bx - A x|| \leq
||B - A || \, || x ||$. If $||x|| = 1$, $ ||A x|| - ||B - A ||  \leq ||Bx||$. So,

\centerline{$\text{inf}\{\|Ax\|:\|x\|=1, x \in {\rm Ker}(B)^{\perp}\}-
||B - A || \leq \gamma(B)$.}

Since ${\rm Ker}(A) \subseteq {\rm Ker}(B)$, then
$\text{inf}\{\|Ax\|:\|x\|=1, x \in {\rm Ker}(B)^{\perp}\}\geq \gamma(A)$. Hence we obtain the conclusion.
\end{proof}
\begin{cor}\label{C An A gamma mayor 0}
Let $A_{n}, A \in B(\mathcal{H},\mathcal{K})$ where $A$ is
invertible and $A_{n} \rightarrow A$ as $n \rightarrow \infty$. Then
there exists $N \in \mathbb{N}$ such that if $n \geq N$ then
$\gamma(A_{n}) > 0$.
\end{cor}
\begin{proof}
If $A$ is invertible then ${\rm Ker}(A)=\{0\}$ and
$\gamma(A)
> 0$. Let $N \in \mathbb{N}$ such that $||A_{n} - A ||
< \gamma(A)$ if $n \geq N$. By Lemma~\ref{L A B}, $\gamma(A_{n})
> 0$ if $n \geq N$.
\end{proof}
\begin{thm}\label{T A invertible An AnV}
Let $(\mathbf{V},{\bf v})$ be a fusion Bessel sequence for $\Vsb$. Given $ n\in\N$, let $A_{n}, A \in B(\mathcal{H})$ where $A$ is
invertible, $A_{n} \rightarrow A$ as $n \rightarrow \infty$, and ${\rm R}(A), {\rm R}(A_{n}) \subseteq \mathcal{V}$. Let $\widetilde{{\bf v}}=\{\widetilde{v}_{k}\}_{k \in I}$ be a family of weights. If there exists $\delta > 0$ such that $\delta \leq v_{i}/\widetilde{v}_{i}$ for all $i \in
I$, then there exists $n_{0} \in \mathbb{N}$ such that
$(A_{n}\mathbf{V},\widetilde{\mathbf{v}})$ is a fusion Bessel sequence for $\mathcal{V}$ for all $n \geq n_{0}$.
\end{thm}
\begin{proof}
Note that ${\rm Ker}(AT_{\mathbf{V},{\bf v}} M_i) = {\rm
Ker}(T_{\mathbf{V},{\bf v}} M_i) \subset {\rm
Ker}(A_{n}T_{\mathbf{V},{\bf v}} M_i)$. So, by Lemma~\ref{L A B},

\centerline{$\gamma(AT_{\mathbf{V},{\bf v}} M_i)-
||A_{n}T_{\mathbf{V},{\bf v}} M_i - AT_{\mathbf{V},{\bf v}} M_i||
\leq \gamma(A_{n}T_{\mathbf{V},{\bf v}} M_i)$.}

We have

\centerline{$||A_{n}T_{\mathbf{V},{\bf v}} M_i - AT_{\mathbf{V},{\bf
v}} M_i|| \leq ||A_{n}- A || \, ||T_{\mathbf{V},{\bf v}} M_i||=
v_{i}||A_{n}- A ||$}

\noindent and by \cite[Remark 2.4 (4)]{Ruiz-Stojanoff (2008)} and
Lemma~\ref{L gamma TVv Mk}

\centerline{$\gamma(A T_{\mathbf{V},{\bf v}} M_i) \geq
\|A\|^{-1}\gamma(T_{\mathbf{V},{\bf v}} M_i)=v_{i}\|A\|^{-1}$.}

\noindent Hence

\centerline{$\gamma(A_{n}T_{\mathbf{V},{\bf v}} M_i) \geq
v_{i}(\|A\|^{-1}-||A_{n}- A ||)$.}

Let $\|A\|^{-1} > \delta_{1} > 0$ and $n_{0}$ such that $||A_{n}- A ||
\leq \|A\|^{-1} - \delta_{1}$ if $n \geq n_{0}$. Then

\centerline{$\gamma(A_{n}T_{\mathbf{V},{\bf v}} M_i) \geq \delta_{1}
v_{i}$ if $n \geq n_{0}$.}

\noindent From here, if $n \geq n_{0}$,

\centerline{$\left(\frac{\gamma(A_{n} T_{\mathbf{V},{\bf v}}
M_i)}{\widetilde{v}_{i}}\right)^{2} \geq \left(\frac{v_{i}\delta_{1}}{\widetilde{v}_{i}}\right)^{2}\geq \left(\delta
\delta_{1}\right)^{2}$.}

Consequently, by Lemma~\ref{L V Bessel},
$(A_{n}\mathbf{V},\widetilde{\mathbf{v}})$ is a fusion Bessel
sequence for $\mathcal{V}$ if $n \geq n_{0}$.
\end{proof}

\section*{Acknowledgement}

The research of J. P. D\'{i}az, S. B. Heineken, and P. M. Morillas has been partially
supported by Grants PIP 112-201501-00589-CO (CONICET) and PROIPRO
3-2823 (UNSL). S. B. Heineken also acknowledges the support of
Grants PICT-2011-0436 (UBA) and UBACyT 20020130100422BA.

% ----------------------------------------------------------------
%\bibliographystyle{amsplain}

\end{document}